\documentclass[reqno,11pt]{amsart}
\usepackage{amsmath,amsthm,amsfonts,color,graphicx}
\usepackage[latin1]{inputenc}
\usepackage[makeroom]{cancel}
\usepackage{tikz}
\usepackage{subfigure,multirow,float}
\usetikzlibrary{decorations.pathreplacing}
\usepackage{filecontents}
\usepackage{appendix,soul}
\usetikzlibrary{plotmarks}
\usepackage{silence}
\definecolor{lightGray}{RGB}{235,235,235}
\definecolor{orange}{RGB}{255,128,0}
\definecolor{ucib}{RGB}{0,36,105}
\definecolor{mygreen}{RGB}{0,128,0}
\definecolor{lightBlue}{RGB}{102,153,204}

\newtheorem{thm}{Theorem}[section]
\newtheorem{lem}[thm]{Lemma}

\newtheorem{prop}[thm]{Proposition}
\newtheorem{rems}[thm]{Remarks}
\newtheorem{rem}[thm]{Remark}

\newtheorem{deff}[thm]{Definition}

\DeclareMathAlphabet{\mathpzc}{OT1}{pzc}{m}{it}

\numberwithin{equation}{section}

\begin{document}
\bibliographystyle{plain}

\title{On the Maximal Parameter Range of Global Stability for a Nonlocal Thermostat Model}

\author{Patrick Guidotti}
\author{Sandro Merino}
\address{University of California, Irvine\\
Department of Mathematics\\
340 Rowland Hall\\
Irvine, CA 92697-3875\\ USA }

\address{Basler Kantonalbank \\
Brunngaesslein 3\\ 
CH-4002 Basel\\
Switzerland}
\email{gpatrick@math.uci.edu and sandro.merino@bkb.ch}

\begin{abstract}
The global asymptotic stability of the unique steady state of 
a nonlinear scalar parabolic equation with a nonlocal boundary
condition is studied. The equation describes the evolution
of the temperature profile that is subject to a feedback control loop.
It can be viewed as a model of a rudimentary thermostat, where a
parameter controls the intensity of the heat flow in response to the
magnitude of the deviation from the reference temperature at a
boundary point. The system is known to undergo a Hopf bifurcation when
the parameter exceeds a critical value. Results on the characterization of
the maximal parameter range where the reference steady state is
globally asymptotically stable are obtained by analyzing a closely
related nonlinear Volterra integral equation. Its kernel is derived
from the trace of a fundamental solution of a linear heat equation.
A version of the Popov criterion is adapted and applied to the
Volterra integral equation to obtain a sufficient condition for the
asymptotic decay of its solutions.        
\end{abstract}

\subjclass[1991]{}

\maketitle

\section{Introduction}
The nonlocal nonlinear problem
\begin{equation}\label{tsEq}
  \begin{cases}
    u_t-u_{xx}=0&\text{in }(0,\infty)\times(0,\pi),\\
    u_x(t,0)=\tanh \bigl( \beta u(t,\pi) \bigr)&\text{for }t\in(0,\infty),\\
    u_x(t,\pi)=0&\text{for }t\in(0,\infty),\\
    u(0,\cdot)=u_0&\text{in }(0,\pi),
  \end{cases}
\end{equation}
with parameter $\beta\in\mathbb{R}$ was first introduced in \cite{GM97} as a simple model for a thermostat
where the sensor is not placed in the same position as the
heating/cooling actuator that receives temperature feedback. 
In that paper, it was shown that the local exponential stability of
the trivial solution $u(t,\cdot)\equiv 0$ is lost when
$\beta\in(0,\infty)$ exceeds the critical value 
$\beta_0 \approx 5.6655$. 
In fact, in \cite{GM97} it is shown that a Hopf bifurcation occurs at
$\beta_0$ which produces a local branch of periodic solutions for
$\beta\in(\beta_0,\beta_0+\varepsilon)$ and some $\varepsilon>0$. In this
paper we show that the trivial solution is globally attractive for $\beta\in(0,\beta_0)$. 
Following the terminology used in \cite{La91} we need to distinguish between the concept 
of the attractor $\hat{\mathcal{A}}_{\beta}$ and the B-attractor $\mathcal{A}_{\beta}$ when 
formulating our main results in Theorem \ref{mainresult}. 
We point out that, in general, the attractor $\hat{\mathcal{A}}_{\beta}$
is a proper subset of the B-attractor $\mathcal{A}_{\beta}$. 
We  prove that the continuous semiflow $\Phi_\beta$ induced on
$\operatorname{H}^1(0,\pi)$ by the system \eqref{tsEq} has a global 
attractor $\hat{\mathcal{A}}_{\beta}=\{ 0\}$ for $\beta\in (0,\beta_0)$ and 
that for $\beta\in (0,\frac{4}{\pi})$ the B-attractor and the attractor coincide, i.e. 
$\mathcal{A}_{\beta}=\hat{\mathcal{A}}_{\beta}=\{0\}\,$.
In \cite{CBOP03}, the authors prove that the global B-attractor $\mathcal{A}_{\beta}$ exists for
$\beta\in(0,\infty)$ and that $\mathcal{A}_{\beta}=\{0\}$ for $\beta\in (0,\frac{1}{\pi})\,$. 
We thus extend previous results by determining larger parameter ranges where the 
B-attractor and the attractor are shown to be equal to $\{0\}\,$.\\ 
The existence of the B-attractor $\mathcal{A}_{\beta}$ shown in \cite{CBOP03} for 
$\beta\in(0,\infty)$ implies, in particular, that all orbits are bounded in the underlying Banach space. 
Since $\operatorname{H}^1(0,\pi) \hookrightarrow \operatorname{C} \Bigl( [0,\pi]\Bigr)$, the orbits satisfy
$$
 \| \Phi _\beta (t,u_0)\| _\infty\leq c(u_0)<\infty,\: t\geq 0.
$$
By means of a weak formulation we interpret \eqref{tsEq} as the
abstract Cauchy problem
\begin{equation}\label{tsACP}
  \begin{cases}
    \dot u+A_Nu=-\gamma_0' \Bigl( \tanh\bigl( \beta \gamma_\pi
    u(t)\bigr)\Bigr) ,&t>0,\\ u(0)=u_0&
  \end{cases}
\end{equation}
in the Banach space $\operatorname{H}^{-1}=\Bigl(
\operatorname{H}^{1}(0,\pi)\Bigr) '$, where $\operatorname{H}^{-1}$ is
the dual space of $\operatorname{H}^{1}:=\operatorname{H}^{1}(0,\pi)$.
Here $A_N$ denotes the unbounded operator
$$
 A_N:\operatorname{dom}(A_N)=\operatorname{H}^1\subset
 \operatorname{H}^{-1}\to\operatorname{H}^{-1}
$$
given by
$$
 \langle A_Nu,v \rangle =\int_0^\pi u_xv_x\, dx,\: u,v\in
 \operatorname{H}^{1},
$$
and the linear operators $\gamma_\pi \in \mathcal{L}
(\operatorname{H}^1,\mathbb{R})$ with $\gamma_\pi(u)=u(\pi)$ and
$\gamma_0'\in \mathcal{L} (\mathbb{R},\operatorname{H}^{-1})$ with
$\langle \gamma_0'(s),v \rangle =s \gamma_0(v)$, for $s\in \mathbb{R}$
and $v\in \operatorname{H}^1$ denote the trace operator at $x=\pi$ and
the dual of the trace operator at $x=0$, respectively. Notice, in
particular, that $\gamma_0'(s)=s\delta_0$, where 
$\delta_0 \in \operatorname{H}^{-1}$ is the Dirac distribution supported in $x=0$, 
i.e.  $\delta_0$ is defined by the duality pairing 
$\left\langle \delta_0, \varphi
\right\rangle_{\operatorname{H}^{-1},\operatorname{H}^{1}}:=\varphi(0)$.

The solution of \eqref{tsACP} exists globally and can be represented
by the variation of constants formula
\begin{align}\label{vcf}
  u(t,u_0)&=e^{-tA_N}u_0 - \int _0^t e^{-(t-\tau)A_N} \gamma_0' \Bigl[
  \tanh\bigl( \beta \gamma_\pi\bigl(  u(\tau,u_0)\bigr)
  \bigr) \Bigr]\, d\tau \notag\\
  &=e^{-tA_N}u_0 - \int _0^t \tanh\Bigl[\beta\gamma_\pi\bigl(
u(\tau,u_0)\bigr)\Bigr]e^{-(t-\tau)A_N}\delta_0\, d\tau 
\end{align}
where the second identity holds due to the linearity of the semigroup 
$\{ e^{-tA_N}\, |\, t\geq 0\}$ generated by the operator $-A_N$ on 
$\operatorname{H}^{-1}$, which is analytic and strongly continuous. 
As described in \cite{GM97} and \cite{GM99} the general
theory presented in \cite{Ama95} allows one to formulate the Cauchy Problem
\eqref{tsACP} on a scale of Banach spaces which includes e.g.
$\operatorname{H}^1\subset\operatorname{H}^{-1}$. As discussed in
\cite{GM97},  and further explored in \cite{CBOP03}, the global semiflow
$\Phi_\beta$ induced by \eqref{tsACP} consists of classical solutions
$u(t,u_0)(x)$ of the parabolic initial boundary value problem
\eqref{tsEq}.

We conclude this section with a brief and incomplete overview on the research 
that the seemingly innocuous thermostat model triggered since its introduction in
\cite{GM97}. In that paper a brief passage in N. Wiener's book ``Cybernetics'' 
is quoted and indicated as the source of initial inspiration. 

Results about the linear(ized) model have been
discussed in \cite{KmcK04} also allowing the sensor location to be
any $x\in(0,\pi]$. They follow an approach based on an integral
reformulation of the problem via the Laplace transform. 
They observe that the same phenomena rigorously proven 
in \cite{GM97} remain valid if the sensor is placed in a different location 
than the actuator, no matter how close they may be. 
Global stability of the trivial solution has been studied in \cite{CBOP03}, 
where the authors prove the existence of a bounded global attractor 
$\mathcal{A}_\beta$ for any $\beta>0$ and manage to obtain global stability of 
the trivial solution for $\beta\in (0,\frac{1}{\pi})$, showing that 
$\mathcal{A}_\beta=\{0\}$ in that parameter range. 
In Remarks \ref{Ljap1} and \ref{Ljap2}, we give two simple proofs for 
global stability. The first recovers the global stability result for 
$\beta\in(0,\frac{1}{\pi})$ and the second extends the parameter range of 
global stability to $(0,\frac{4}{\pi})$. 

Problem \eqref{tsEq} has also been studied in the presence of noise in \cite{L18}. 
There the author shows that the average solution of the thermostat problem 
with randomly switching locations of the sensor and of the actuator can 
exhibit exponential growth in spite of the fact that stability holds for both 
configurations. 

The observation that the first eigenfunction of the 
linearization of the stationary problem remains positive up to 
$\beta=\frac{1}{2}$ and that, for $\beta\in(0,\frac{1}{2})$, 
the semi-group is non-positive leads to the concept of eventually positive semi-groups. 
These ideas have recently been introduced and developed in \cite{DGK16,DG18}. 

Various nonlinear stationary problems associated with \eqref{tsEq} have been 
studied by several authors. We only mention G. Infante and J.R.L Webb 
(\cite{GeWe06(1)}, \cite{GeWe06(2)}) here and refer to their bibliography for numerous additional results.      

The Hopf bifurcation phenomenon engendered by the nonlocal nature of the 
boundary condition has also inspired the research presented in \cite{GoGuSo13} on a 
market price formation model introduced by J.M. Lasry and P.L. Lions. 
In particular, in that specific context a similar Hopf bifurcation scenario 
shows that ``demand" and ``supply" do not simply create stable prices but can 
lead to price oscillations. The phenomenon emerges on the basis of the 
modelled behaviour of the population densities of buyers and sellers 
positioned in a liquid market over a continuum of prospective transaction prices.

The variety of results obtained since the introduction of this 
rudimentary thermostat model for the original purpose of  
underpinning N. Wiener's postulated ``wild oscillations" more 
explicitly and rigorously, bears witness to the interesting 
underlying mathematical properties of this prototypical ``parabolic oscillator".

Before proceeding to the derivation of our main result in the next 
section, we provide two simpler proofs, albeit at the cost of only
obtaining smaller ranges for global stability.  
In the first remark it is shown that the stability result of
\cite{CBOP03} can be obtained in a different
way that gives additional insight into the behavior of
solutions. In Remark \ref{Ljap2} we provide a sharper proof that extends the 
stability range from $(0,\frac{1}{\pi})$ to $(0,\frac{4}{\pi})\,$. 

\begin{rem}\label{Ljap1}
A simple calculation shows that
$$
 \frac{1}{2}\frac{d}{dt}\int _0^\pi u^2(t,\xi)\, d\xi=
 -\int _0^\pi u_x^2( t,\xi)\, d\xi-u(t,0)\tanh \bigl(\beta
 u(t,\pi)\bigr),
$$
so that $\| u(t,\cdot)\| _2$ can only grow if
$u(t,0)u(t,\pi)<0$. An application of the Cauchy-Schwarz inequality yields that
$$
 \bigl( u(t,0)-u(t,\pi)\bigr)^2=\bigl( \int _0^\pi u_x(t,x)\, dx\bigr)^2\leq \pi\, \int _0^\pi u_x^2(t,\xi)\,
 d\xi.
$$
It follows that
$$
 \frac{1}{2}\frac{d}{dt}\int _0^\pi u^2(t,\xi)\, d\xi\leq -u(t,0)\tanh
 \bigl(\beta u(t,\pi)\bigr)-\frac{1}{\pi}\bigl( u(t,0)-u(t,\pi)\bigr)^2,
$$
and that $\frac{d}{dt}\| u(t,\cdot)\|^2_2$ can only be non-negative if
$$
 \frac{1}{\pi}\bigl( u(t,0)-u(t,\pi)\bigr)^2\leq \tanh \bigl( \beta
 \big |u(t,\pi)\big |\bigr)\big |u(t,0)\big |.
$$
Therefore for non-negativity it is necessary that
$$
 \frac{1}{\pi}\big |u(t,0)\big |\leq \tanh \bigl( \beta
 \big |u(t,\pi)\big |\bigr), 
$$
i.e., that $\big |u(t,0)\big |\leq \pi$, and that
$$
 \frac{1}{\pi}\big |u(t,\pi)\big |^2\leq\tanh \bigl( \beta
 \big |u(t,\pi)\big |\bigr)\big |u(t,0)\big |\leq \pi,
$$
i.e., that $\big |u(t,\pi)\big |\leq \pi$, but also that
$$
 \big |u(t,\pi)\big |^2\leq \pi^2\tanh^2 \bigl( \beta
 \big |u(t,\pi)\big |\bigr),
$$
which entails that
$$
 \big |u(t,\pi)\big |\leq \pi\tanh\bigl( \beta
 \big |u(t,\pi)\big |\bigr)\leq \pi\beta\big |u(t,\pi)\big |,
$$
or, equivalently, that $\beta\geq\frac{1}{\pi}$. This yields the
stability of the trivial solution for $\beta\in(0,\frac{1}{\pi})$ but also
the boundedness of the orbits. We don't give the details of the argument
and refer to the result on the existence of the global B-attractor for any 
$\beta>0$ which is proven in \cite{CBOP03} and motivated us to revisit this problem. 
\end{rem}
In the next remark we extend the stability range to $\beta \in (0\,,\frac{4}{\pi})\,$ 
by a relatively simple and direct argument.
\begin{rem}\label{Ljap2}
Rewrite \eqref{tsACP} as
$$
 u_t+A_Nu=-\frac{\tanh(\beta z)}{z}\gamma_\pi(u) \delta_0,
$$
for $z(t)=\gamma_\pi \bigl( u(t,\cdot)\bigr)$. Notice that
$\psi(z)(t)=\frac{\tanh\bigl(\beta z(t)\bigr)}{z(t)}\in[\delta,\beta]$ for 
$t\geq 0$ and some $\delta=\delta(u_0)>0$ since solutions of \eqref{tsACP} 
remain pointwise bounded as shown in \cite{CBOP03}. We claim that
$$
 \frac{1}{2}\int _0^\pi u^2(x)\, dx=\frac{1}{2}(u,u)_2
$$
is a Ljapunov functional for the equation as long as 
$\beta \in (0\,,\frac{4}{\pi})\,$. In Section 5 we will discuss that this implies 
that the B-attractor is given by $A_\beta=\{0\}$ for that
range of the parameter $\beta$. In order to verify that we indeed have a Ljapunov functional we
compute
\begin{align*}
 \frac{d}{d t}\frac{1}{2}(u,u)_2&= \Big\langle
  -A_Nu-\psi(z)u(\cdot,\pi)\delta _0,u \Big\rangle
 _{\operatorname{H}^{-1},\operatorname{H}^1} \\
 &= -\big\langle A_Nu,u \big\rangle
   _{\operatorname{H}^{-1},\operatorname{H}^1}
   -\psi(z)u(\cdot,\pi)u(\cdot,0)\\
 &= -\big\langle A_Nu,u \big\rangle
   _{\operatorname{H}^{-1},\operatorname{H}^1} -
   \psi(z)\Big\langle \frac{1}{2}(\delta_0 \delta_\pi^\top +\delta_\pi \delta ^\top_0)u, u
   \Big\rangle _{\operatorname{H}^{-1},\operatorname{H}^1} \\
  &=-\Big \langle A_Nu+\frac{\psi(z)}{2}(\delta_0 \delta_\pi^\top +\delta_\pi \delta
   ^\top_0)u,u \Big \rangle _{\operatorname{H}^{-1},\operatorname{H}^1},
\end{align*}
where we use the suggestive notation $\delta_x^\top:=\gamma_x$ for
$x\in \{0,\pi\}\,$ to highlight the inherent symmetry. Thus we conclude that the orbital 
derivative is strictly negative on non-stationary orbits if the self-adjoint operators
\begin{equation}\label{saoperators}
 A_N+\frac{\alpha}{2}(\delta_0 \delta_\pi^\top +\delta_\pi \delta
 ^\top_0)
\end{equation}
only possess strictly positive eigenvalues for any parameter value  
$\alpha\in(\delta,\beta)$. 
We will now prove that this is the case for $\alpha\in(0,\frac{4}{\pi})\,$.   
The operators in \eqref{saoperators} correspond to the weak formulation of
the homogeneous heat equation in $(0,\pi)$ with boundary conditions
$$
 u_x(t,0)=\frac{\alpha}{2} u(t,\pi),\: u_x(t,\pi)=-\frac{\alpha}{2}
 u(t,0)\text{ for }t>0.
$$
The eigenfunctions to a putative eigenvalue $\lambda^2\in\mathbb{R}$ can 
be assumed to be of the form $A\sin(\lambda x)+B\cos(\lambda x)$. 
In fact, the trace operators added to the Neumann Laplacian in \eqref{saoperators} 
constitute a compact perturbation, as is discussed in \cite{GM99}. 
Therefore the operators \eqref{saoperators} conserve a pure point spectrum.  Imposing 
the boundary conditions leads to a linear system that only admits non-trivial solutions 
if its determinant vanishes. The characteristic equation is easily found to be 
$$
 \lambda-\bigl[\frac{\alpha^2}{4}+\lambda^2\bigr]
 \frac{\sin(\lambda\pi)}{\alpha}=0.
$$
If negative eigenvalues $\lambda ^2$ exist, then they need to be of the form $\lambda=i\mu$ for 
$\mu\in\mathbb{R}\,$. Then $\mu$ is a solution of 
$$
 \mu=\bigl[\frac{\alpha^2}{4}-\mu^2\bigr]
 \frac{\sinh(\mu\pi)}{\alpha}=:r_\alpha (\mu).
$$
Note that for $|\mu|>\frac{\alpha}{2}$, $\operatorname{id}_\mathbb{R}$ 
and $r_\alpha$ have opposite signs. Any non-trivial solution is therefore 
confined to the interval $\bigl[-\frac{\alpha}{2},\frac{\alpha}{2}\bigr]$,
and, owing to the oddity of the functions, we only need to examine the 
positive half of that interval. A simple calculation shows that 
$r_\alpha '(0)=\frac{\pi}{4}\alpha$, so that $r_\alpha '(0)<1$ as long as 
$\alpha<\frac{4}{\pi}\,$. We proceed to show that $r_\alpha$ is
strictly concave in $[0,\frac{\alpha}{2}]$ for $\alpha\in (0,\frac{4}{\pi})\,$.  
As a consequence no solution $\mu$ of the characterisitic equation can exist  
and hence no non-positive eigenvalue of \eqref{saoperators}  
is found for $\alpha\in (0,\frac{4}{\pi})\,$. 
The straightforward computation of the second derivative of $r_\alpha$ yields  
$$
 4\alpha\, r_\alpha ''(\mu)= \sinh(\mu\pi)\bigl[\pi^2(\alpha^2-4\mu^2)
 -8\bigr]-16\pi\mu\cosh(\mu\pi).
$$
To find a criterion that implies a negative second derivative of $r_\alpha$ observe that
\begin{align*}
 \frac{\tanh(\mu\pi)}{\mu}\bigl[ \pi^2(\alpha^2-4\mu^2)-8\bigr]\leq
 \pi\bigl[ \pi^2(\alpha^2-4\mu^2)-8 \bigr]\leq \alpha^2\pi^3-8\pi.
\end{align*}
This leads to the condition $ \alpha^2\pi^3-8\pi<16\pi$ which holds when $\alpha$ is less than 
$\frac{\sqrt{24}}{\pi}>\frac{4}{\pi}\,$. The desired concavity follows and we 
conclude that the operators \eqref{saoperators} have a 
strictly positive point spectrum, whenever $\alpha\in (0,\frac{4}{\pi})\,$. 
To see that this construction of the Ljapunov functional cannot be extended indefinitely, note that, for $\alpha>\frac{4}{\pi}$, 
$r_\alpha'(0)>1$, which together with $r_\alpha(\frac{\alpha}{2})=0$ 
implies that the characterisitic equation $\mu=r_\alpha(\mu)$ has at least one solution in 
$[0,\frac{\alpha}{2}]$ thus producing a negative eigenvalue of \eqref{saoperators}.
\end{rem}
In this paper a different approach is presented that will ensure stability 
of the trivial solution all the way to the Hopf bifurcation from the steady state 
that occurs for $\beta=\beta_0\approx 5.6655\,$. 
The proof will not construct a Ljapunov functional explicitly. It sharpens 
the result in this remark to a maximal range of stability, yet  
at the expense of a significantly more technical proof.
\section{The Volterra Integral Equation at the boundary component $\left\{ x=\pi \right\}$ }
The (abstract) variation of constant formula representation
\eqref{vcf} for the classical solutions of \eqref{tsEq} can be evaluated at $x=\pi$ to verify that for any given initial state $u_0\in \operatorname{H}^1$ 
$$
u(t):=\gamma_\pi \bigl( u(t,u_0) \bigr), 
$$
solves a Volterra integral equation of the second kind. We collect this observation in the next lemma. 

\begin{lem}\label{VoltInt1}
Let $\beta\in(0,\infty)$ and let $\Phi_\beta(\cdot, u_0)$ be any orbit of the continuous semiflow $\bigl(\Phi_\beta,\operatorname{H}^1\bigr)$ then 
$$
 u(t):=\gamma_\pi \bigl( \Phi _\beta (t,u_0) \bigr)=\Phi _\beta (t,u_0)(x=\pi)
$$ 
solves the nonlinear convolution-kernel Volterra integral equation of the 
second kind
\begin{equation}\label{vie}
y(t)=f(t)+\int _0^t a(t-\tau)g_\beta \bigl( y(\tau)\bigr)\, d\tau , t >0,
\end{equation}
where the so-called forcing function $f\equiv f(u_{0})$, the convolution 
kernel $a$, and $g_{\beta}$ are defined as 
\begin{align*}
  f(t)&:=\bigl( e^{-tA_N}u_0\bigr)(\pi)\text{ for }t\geq 0,\\
  a(t)&:=-\bigl( e^{-tA_N}\delta _0 \bigr)(\pi)\text{ for }t > 0,\\
  g_\beta(w)&:=\tanh ( \beta w)\text{ for }w\in \mathbb{R} . 
\end{align*} 

\end{lem}
\begin{proof}
Note that for $t>0$ the analyticity of the semigroup implies that
$e^{-tA_N}\delta _0\in\operatorname{H}^1$ since $\delta_0\in
\operatorname{H}^{-1}$.
Hence, thanks to the embedding $\operatorname{H}^1 \hookrightarrow
\operatorname{C} \bigl( [0,\pi]\bigr)$, the kernel $a$ can be defined
pointwise by evaluating at $x=\pi$.
The evaluation at $x=\pi$ for $t>0$ yields
$$
  u(t,u_0)(x=\pi) =\bigl( e^{-tA_N}u_0\bigr)(x=\pi) - \int _0^t
  \tanh\Bigl[\beta\gamma_\pi\bigl(  u(\tau,u_0)\bigr)\Bigr]\bigl(
  e^{-(t-\tau)A_N}\delta _0 \bigr)(x=\pi)\, d\tau 
$$ 
and concludes the proof.
\end{proof}
Next we collect some remarks, which will be needed in the 
subsequent study of the above Volterra integral equation.
\begin{rems}\label{rems}
{\bf (a)} A full spectral resolution of the operator $A_N$
can be computed in order to obtain a series representations for 
the (kernel of the) semigroup generated by $-A_N$, which reads
\begin{equation}\label{neuFundamental}
 N(t,x):=e^{-tA_N}\delta_0=\sum _{k=0}^\infty c_k\cos(kx)e^{-tk^2}
\end{equation}
and is the fundamental solution for the parabolic homogeneous 
Neumann problem on the interval $[0,\pi]$, where $c_0=\frac{1}{\pi}$
and $c_k=\frac{2}{\pi}$ for $k\geq 1$.
Consequently, the integral kernel $a$ satisifies $a=-N(t,\pi)$, i.e. 
it holds that
\begin{equation}\label{aRep}
  a(t)=-\frac{1}{\pi}+\frac{2}{\pi}\sum_{k=1}^\infty (-1)^{k+1}e^{-tk^2},\: t>0.
\end{equation}
Besides the convolution kernel $a$, the forcing term $f$ can also be
expressed in terms of the basis of eigenfunctions to give
\begin{equation}\label{fRep}
 f(t)=\sum_{k=0}^\infty \langle u_0,\varphi _k \rangle
 e^{-tk^2}\varphi_k(\pi)= \bar u_0+\sqrt{\frac{2}{\pi}}\sum _{k=1}^\infty (-1)^k \hat
 u_{0k}e^{-tk^2},
\end{equation}
for $\hat u_{0k}=\langle u_0,\varphi _k \rangle$ and $k\in \mathbb{N}$.\\[0.1cm]
{\bf (b)} For $u_0\in \operatorname{H}^1$ the forcing function is in
$\operatorname{BUC}^\infty \bigl( (0,\infty)\bigr)$ and its $n$-th
derivative satisfies
$$
 f^{(n)}\in \operatorname{L}^p\bigl((0,\infty)\bigr)\text{ for }p\in [1,\infty] \;,\;n\geq 1\;.
$$
This follows from \eqref{fRep}, i.e. from the exponential convergence
of $u_N(t,u_0)$ to the constant function $\bar u_0$ as $t \to\infty\,$.\\[0.1cm]
{\bf (c)} The fundamental solution \eqref{neuFundamental} can also be obtained from the heat kernel on the whole real line: 
$$
 H(t,x):=\frac{1}{\sqrt{4\pi t}}e^{-\frac{x^2}{4t}},\: x\in\mathbb{R},\: t>0.
$$
First one obtains the kernel $H_\pi$ for the $2\pi$-periodic heat equation by ``periodization''
$$
 H_\pi (t,x):=\sum _{k\in\mathbb{Z}}H(t,x-2\pi k),\: x\in [0,2\pi).
$$
Then by introducing the Riemann theta function
$$
\theta_1(\tau,z):=\sum_{k\in\mathbb{Z}}e^{\pi i k^2 \tau}\,e^{2\pi i kz} \:,\, \tau, z \in \mathbb{C}\,,\, \operatorname{Im}(\tau)>0$$
and by setting
$$
\theta(t,x):=\theta_1(\frac{it}{\pi},\frac{x}{2\pi})=\sum_{k\in\mathbb{Z}}e^{-k^2 t}\,e^{ikz}\:\:, t>0,\, x\in \mathbb{R}\,,
$$
it is directly verfied with \eqref{neuFundamental} that 
\begin{equation}\label{Neumanntheta}
N(t,x)=\frac{1}{\pi}\,\theta(t,x)\:,\, t>0,\: x\in [0,\pi]\:.
\end{equation}
It is a classical result (see e.g. \cite{Lang96}) that the periodic heat kernel 
$H_\pi$ can be represented by $\theta\,$, i.e. 
$$
H_{\pi}(t,x)=\frac{1}{2\pi}\:\theta(t,x)\:,\, t>0\:,\: x\in[0,2\pi]\,.
$$
As a consequence of \eqref{Neumanntheta} the Neumann fundamental solution $N$
can be written in terms of the periodic heat kernel 
\begin{equation}\label{Neumannheat}
N(t,x)=2H_{\pi}(t,x)=H_\pi(t,x)+H_\pi(t,2\pi-x)\:,\, t>0,\: x\in [0,\pi]\:.
\end{equation}
The fundamental solution for the Neumann problem can thus be constructed 
from the heat kernel on $\mathbb{R}\,$ by ``periodization" and ``reflection".  \\[0.1cm]
{\bf (d)} For $a$ and for the ``shifted kernel" $a_s(t):=a(t)+\frac{1}{\pi}$ it holds that
$$
 \lim _{t\to\infty} a(t)=-\frac{1}{\pi},\: \lim _{t\to\infty} a_s(t)=0,
$$
and for $n \geq 1$ 
$$
 \lim _{t\searrow 0} a(t)=\lim _{t\searrow 0} a_s^{(n)}(t)=0\:.
$$
Also note that, by definition,
$$
a^{(n)}=a_s^{(n)},
$$
for $n\geq 1$. The limits for $t\to\infty$ are obtained directly from \eqref{aRep}. 
To determine the one-sided limits we observe that    
$$
-a(t)=N(t,\pi)=2H_{\pi}(t,\pi) 
$$
and therefore the limits for ${t\searrow 0}$ follow from the well-known properties of the heat kernel on the unit circle. 
In particular from the concentration of the kernel's mass at $x=0$ as ${t\searrow 0}\,$, more precisely  
$$
 \lim _{t\searrow 0}\: H_{\pi}(t,x) = \lim _{t\searrow 0}\: H^{(n)}_{\pi}(t,x)=0 \:, n\geq 1 \:,\: x\in (0,2\pi)\;.
$$
In the sequel we will not need the one-sided limits for the derivatives $a^{(n)}\,$. 
For the sake of being more self-contained we also provide an alternative, more direct 
argument for $\lim _{t\searrow 0} a(t)=0\,$. The Jacobi theta function   
$$
\vartheta_4(z,q):=\sum_{k=-\infty}^{\infty}(-1)^k q^{k^2}e^{2kiz}=1+2\sum_{k=1}^{\infty}(-1)^k q^{k^2}cos(2kz)
$$  
is particularly simple for $z=0$  
$$
\vartheta_4(q):=\vartheta_4(0,q)=1+2\sum_{k=1}^{\infty}(-1)^k q^{k^2}.
$$
It is known ( \cite{Bor87}, Chapter 3 ) that $\lim_{q\nearrow 1} \vartheta_4(q) = 0$ 
and therefore 
\begin{equation}\label{thetalim}
\lim_{q\nearrow 1} \sum_{k=1}^{\infty}(-1)^k q^{k^2} =-\frac{1}{2}\,.
\end{equation}
Therefore by setting $q=e^{-t}$ we obtain $\lim _{t\searrow 0} \, a(t)=0$ from  
$$
\lim_{t\searrow 0} \, \frac{2}{\pi}\sum_{k=1}^{\infty}(-1)^{k+1} \, e^{-tk^2} = \frac{1}{\pi}
$$
and \eqref{aRep}.\\[0.1cm]
{\bf (e)} The kernel $a$ (extended by 0 for $t\leq0$) satisfies
$$
 a\in \operatorname{BC}^\infty(\mathbb{R},\mathbb{R})
$$
and its $n$-th derivatives satisfy
$$
 a^{(n)}\in \operatorname{L}^p(\mathbb{R}) \:, p\in [1,\infty]\:,\: n\geq1\;.
$$
This follows from the kernel's representation \eqref{aRep}.\\[0.1cm]
%
{\bf (f)} Similarly the shifted kernel $a_s=a+\frac{1}{\pi}\,$, extended smoothly to $t=0$ by setting $a_s(0)=\frac{1}{\pi}\,$, satisfies
$$
 a_s\in \operatorname{BC}^\infty([0,\infty),\mathbb{R})\cap
 \operatorname{L}^p\bigl((0,\infty)\bigr) \;\;,\; p\in [1,\infty]\;.
$$
Again, this follows from  \eqref{aRep}. Note that the extension of $a_s$ to negative 
arguments by zero creates a discontinuity at $t=0\,$. This will be relevant when 
computing its Fourier transform.\\[0.1cm]
%
%
{\bf (g)} The series representation of the Fourier transforms of $a_s$
and $a_s':=a_s^{(1)}$ will be needed in the sequel to recover the Popov stability 
criterion from the Volterra integral equation. They are given by
\begin{equation}\label{asFourier}
  \hat a_s(\omega)=\frac{2}{\pi}\sum_{k=1}^\infty(-1)^{k+1}
  \frac{k^2-i \omega}{k^4+\omega ^2}
\end{equation}
and by
\begin{equation}\label{as'Fourier}
  \widehat{a_s'}(\omega)=-\frac{1}{\pi}+\frac{2}{\pi}\sum_{k=1}^\infty(-1)^{k+1}
  i \omega\frac{k^2-i \omega}{k^4+\omega ^2},
\end{equation}
respectively. Note that, here, we define $a_s(t)$ also for negative arguments by 
setting its value to zero for $t<0\,$.  As a consequence of this extension, 
$a_s$ is discontinuous at $t=0\,$, which impacts the Fourier transform of its distributional 
derivative $a_s'\,$. \\
To determine the Fourier transforms we use \eqref{aRep} and interchange 
the order of summation and integration, which is justified by the uniform 
convergence in \eqref{aRep}. To compute $\hat a_s$ observe that 
$$
\hat a_s(\omega)=\int_{\mathbb{R}}e^{-i\omega t}a_{s}(t)dt=\int_{0}^{\infty}e^{-i\omega t}a_{s}(t)dt
$$
and therefore 
$$
\hat a _s(\omega) = \frac{2}{\pi}\sum_{k=1}^\infty (-1)^{k+1}\int_{0}^{\infty} 
e^{-i\omega t} e^{-tk^2}\,dt=\frac{2}{\pi}\sum_{k=1}^\infty (-1)^{k+1}
\int_{0}^{\infty} e^{-(i\omega+k^2)t} \,dt=\frac{2}{\pi}\sum_{k=1}^\infty 
(-1)^{k+1}\frac{1}{k^2+i\omega}\,.
$$
To find the Fourier transform of $a_{s}'$ one can proceed similarly and we 
evaluate the occuring integral as follows  
\begin{align*}
{\widehat {a'_s}}(\omega)&=\int_{0}^{\infty}e^{-i\omega t}a_{s}'(t)dt=
\lim_{\varepsilon\to 0}\Big\{\frac{2}{\pi}\sum_{k=1}^\infty (-1)^{k+2}
\int_{\varepsilon}^{\infty} k^2 e^{-k^2t} e^{-i\omega t} \,dt\Big\}\\
&=\lim_{\varepsilon\to 0}\Big\{\frac{2}{\pi}\sum_{k=1}^\infty (-1)^{k+1}
e^{-k^2t} e^{-i\omega t}\Big |_\varepsilon^\infty\Big\}-i\omega\frac{2}{\pi}
\sum_{k=1}^\infty \int _0^\infty(-1)^{k}e^{-k^2t} e^{-i\omega t}\, dt.
\end{align*}
By using \eqref{thetalim} we find the result  
$$
{\widehat {a'_s}}(\omega)=-\frac{1}{\pi}+\frac{2}{\pi}\sum_{k=1}^\infty(-1)^{k+1}
i\omega\frac{(k^2-i\omega)}{k^4+\omega^2}.
$$
%
%
{\bf (h)} The Laplace transform $\mathcal{L}(a)$ of $a$ is given by
$$
\mathcal{L} (a)(s)=-\frac{1}{\pi s}-\frac{1}{2\pi}\sum _{k=1}^\infty (-1)^k\frac{1}{s+k^2}
\:,\:s\in\{ z\in \mathbb{C} \, |\, \operatorname{Re}(z)>0\}\;.
$$
The Laplace transform  $\mathcal{L}(a)$ of the kernel also has the closed form representation
$$
 \mathcal{L} (a)(s)=-\frac{1}{\sqrt{s}\sinh(\pi\sqrt{s})}\,.
$$
The series representation of the Laplace transform is obtained from \eqref{aRep} by 
elementary integrations. Its explicit representation is obtained 
in \cite{GM99} (formula (10)) in a different context and is derived again 
in Section \ref{NyqLinStab}. A more elementary direct computation 
using a partial fraction expansion is also given in e.g. \cite{CM09}. \\[0.1cm]
%
%
{\bf (i)} We note that $G_+(s):=-\mathcal{L}(a)$ can be expressed in
terms of the Fourier transform of $a_s$ and that of $a_s'\,$. 
In fact, for
$\omega\in \mathbb{R}$, it holds that
$$
 \operatorname{Re}\bigl( G_+(i \omega)\bigr)=-\operatorname{Re}
 \bigl( \hat a_s(\omega)\bigr)
$$
and 
\begin{equation*}
 \omega \operatorname{Im}\bigl( G_+(i \omega)\bigr)=
 \operatorname{Re}\bigl( \widehat{a'_s}(\omega)\bigr) .
\end{equation*}
The inequality
\begin{equation}\label{M}
 \operatorname{Re}\bigl( \hat a_s(\omega)\bigr)+
 q\operatorname{Re}\bigl( \widehat{a'_s}(\omega)\bigr)-\frac{1}{\beta}<0 \,,\omega\in \mathbb{R}\,,
\end{equation}
is then equivalent to
\begin{equation}\label{P}
  \operatorname{Re}\bigl( G_+(i \omega)\bigr)-q \omega
  \operatorname{Im}\bigl( G_+(i \omega)\bigr)>-\frac{1}{\beta} \,,\omega\in \mathbb{R}\,.
\end{equation}
We will use this relationship between $\mathcal{L}(a)$ and
$\hat a_s$ and $\widehat{a'_s}$ to verify that the stability condition
\eqref{M} obtained from the analysis of the integral equation
\eqref{vie} is precisely the Popov stability criterion \eqref{P} applied to the
transfer function $G_+\,$.
\end{rems}

The next proposition is one key ingredient for our main stability
result. It shows that, in order to infer the decay of all orbits of
the semiflow $\bigl(\Phi_\beta,\operatorname{H}^1\bigr)$ to zero in
the Banach space $\operatorname{H}^1$, it is sufficient to prove that
their trace at $x=\pi$ decays to zero in $\mathbb{R}$. Hence the
stability analysis of the nonlinear parabolic evolution problem \eqref{tsEq}
reduces to the study of the aymptotic behaviour as $t\to \infty$ of
the solutions of the Volterra integral equation \eqref{vie}. It turns
out that the nonlocality of the problem is thereby encoded in the
kernel $a$ in a more tractable way. 

\begin{prop}\label{ConvEquiv}
For fixed $\beta\in(0,\infty)$ consider orbits $\Phi_\beta(\cdot, u_0)$ of
the semiflow $\bigl(\Phi_\beta,\operatorname{H}^1\bigr)$. Then, for any
$u_{0}\in \operatorname{H}^1$, it holds that 
$$
 \Phi _\beta (t,u_0)\longrightarrow 0 \text{ in } \operatorname{H}^1
 \iff \gamma_\pi \bigl( \Phi_\beta(t,u_0)\bigr) \longrightarrow 0
 \text{ in } \mathbb{R}\,,
$$
as $t\to\infty\,$. 
\end{prop}
\begin{proof}
  ``$\Rightarrow$'': If $\Phi _\beta (t,u_0)\to 0$ as $t\to\infty$, then
  the fact that $\gamma_\pi\in \mathcal{L}(\operatorname{H}^1,\mathbb{R})$ 
  implies the continuity of $\gamma_\pi$ and thus 
  $\gamma_\pi \bigl( \Phi_\beta(t,u_0)\bigr)\to 0$ as $t\to\infty$.\\[0.1cm]
  ``$\Leftarrow$'': If $u(t,\pi)\to 0$ as $t\to\infty$, then by
  \eqref{vcf} and \eqref{vie} we have that
  $$
   \lim _{t\to\infty}\Bigl[ u_N(t,\pi)+\int _0^t a(t-\tau)g_\beta
   \bigl( u(\tau,\pi)\bigr)\, d\tau \Bigr]=0
  $$
  which entails the convergence of the integral addend since
  $u_N(t,\pi)\to \bar u_0$ as $t\to\infty$. We claim that the limit
  can be identified, i.e. that
  $$
   \lim _{t\to\infty}\int _0^t a(t-\tau)g_\beta
   \bigl( u(\tau,\pi)\bigr)\, d\tau=-\frac{1}{\pi}\int_0^\infty
   g_\beta\bigl( u(\tau,\pi)\bigr)\, d\tau.
  $$
  Indeed, by assumption, given any $\varepsilon>0$, a time
  $t_\varepsilon>0$ can be found such that
  $$
   \big |g_\beta\bigl( u(\tau,\pi)\bigr)\big |\leq \varepsilon \text{ for
   }\tau\geq t_\varepsilon,
  $$
  since $g_\beta(0)=0$ and $g_\beta$ is continuous. Then, for any
  $t>t_\varepsilon$ we have that
  $$
  \int _0^t a(t-\tau)g_\beta\bigl( u(\tau,\pi)\bigr)\, d\tau=
  -\frac{1}{\pi}\int _0^t g_\beta\bigl( u(\tau,\pi)\bigr)\, d\tau+
  \int _0^t [a(t-\tau)+\frac{1}{\pi}]g_\beta\bigl( u(\tau,\pi)\bigr)\, d\tau.
  $$
  Next we use the fact that we know the full spectral resolution of
  the operator $A_N$ given by
  $$
   A_Nu=\sum _{k=1}^\infty k^2 \underset{\hat u_k}{\underbrace{\langle u,\varphi_k\rangle}}\varphi _k,\:
   u\in \operatorname{dom}(A_N),
  $$
  where the eigenfunctions are given by $\varphi
  _k(x)=\sqrt{\frac{2}{\pi}}\cos(kx)$ for $k\in \mathbb{N}$ and $x\in [0,\pi]$ and
  $\varphi_0\equiv\frac{1}{\sqrt{\pi}}$. This yields a representation of
  the semigroup as
  $$
   e^{-tA_N}u_0=\bar u_0+\sum _{k=1}^\infty\hat u_{0k}e^{-tk^2}\varphi_k.
  $$
  Now observe that
  $$
   \delta_0=\frac{1}{\pi}+\sum _{k=1}^\infty \langle \delta_0, \varphi_k \rangle
   \varphi _k=\frac{1}{\pi}+\sqrt{\frac{2}{\pi}}\sum _{k=1}^\infty\varphi _k
  $$
  with convergence in $\operatorname{H}^{-1}$, which yields
  $$
  a(t)=-\bigl( e^{-tA_N}\delta _0\bigr)(\pi)=
  -\frac{1}{\pi}-\sqrt{\frac{2}{\pi}}\sum_{k=1}^\infty e^{-tk^2}\varphi_k(\pi)
  = -\frac{1}{\pi}+\frac{2}{\pi}\sum_{k=1}^\infty(-1)^{k+1} e^{-tk^2}.
  $$
  Next we split the integral
  $$
   \int _0^t [a(t-\tau)+\frac{1}{\pi}]g_\beta\bigl( u(\tau,\pi)\bigr)\, d\tau
  $$
  into the integral up to $t_\varepsilon$ and the rest. Estimating
  separately, we see that
  \begin{align*}
   \Big | \int _0^{t_\varepsilon} g_\beta \bigl(
   u(\tau,\pi)\bigr)\bigl[ a(t-\tau)+\frac{1}{\pi}\bigr] \, d\tau\Big |
   &\leq\frac{2}{\pi}\sum _{k=1}^\infty \int _0^{t_\varepsilon}e^{-(t-\tau)k^2}\,
   d\tau\\&\leq\frac{2}{\pi}\sum _{k=1}^\infty
   \frac{1}{k^2}e^{-(t-t_\varepsilon)k^2}
   \leq c e^{-(t-t_\varepsilon)}, \: t>t_\varepsilon,
  \end{align*}
  since $\| g_\beta\|_\infty \leq 1$, and that
  \begin{align*}
   \Big | \int _{t_\varepsilon}^t g_\beta \bigl(
   u(\tau,\pi)\bigr)\bigl[ a(t-\tau)+\frac{1}{\pi}\bigr] \, d\tau\Big |
   &\leq c\sum _{k=1}^\infty \int _{t_\varepsilon}^t\big | g_\beta \bigl(
   u(\tau,\pi)\bigr)\big |e^{-(t-\tau)k^2}\,
   d\tau\\&\leq c\,\varepsilon \sum _{k=1}^\infty \frac{1}{k^2} \bigl(
   1-e^{-(t-t_\varepsilon)k^2}\bigr)\leq c\,\varepsilon , \:
     t>t_\varepsilon.
  \end{align*}
  This allows us to conclude that, given $\varepsilon>0$, there is
  $\tilde t_\varepsilon>0$ such that
  $$
   \Big |\int _0^t a(t-\tau)g_\beta
   \bigl( u(\tau,\pi)\bigr)\, d\tau+\frac{1}{\pi}\int_0^t
   g_\beta\bigl( u(\tau,\pi)\bigr)\, d\tau\Big |\leq \varepsilon 
  $$
  for $t>\tilde t_\varepsilon$, which yields the stated
  convergence. Next notice that, given $x\in [0,\pi)$, it holds that
  $$
   u(t,x)=u_N(t,x)+\int _0^t a(t-\tau,x)  g_\beta\bigl(
   u(\tau,x)\bigr)\, d\tau,
  $$
  where
  $$
  a(t,x)=-\bigl( e^{-tA_N}\delta _0 \bigr)(x)=
  -\frac{1}{\pi}+\frac{2}{\pi}\sum_{k=1}^\infty \cos(kx)e^{-tk^2}
  \text{ for }t\geq 0.
  $$
  A similar argument then shows that also for $x\in[0,\pi)$ and $t\to\infty$
  $$
   u(t,x)\to \bar u_0-\frac{1}{\pi}\int _0^\infty g_\beta\bigl(
   u(\tau,\pi)\bigr)\, d\tau=0\;. 
  $$
  The limit is $0$ since we already proved that
  $$
   \bar u_0=\frac{1}{\pi}\int _0^\infty g_\beta\bigl(
   u(\tau,\pi)\bigr)\, d\tau
  $$
  in the first step of the proof. 
  Note that, while $a(t,0)$ has a singularity in $t=0$, this singularity is integrable, 
  as shown in Remark \ref{rems}(c), and the argument
  goes through.
  We have thus shown that $\Phi _\beta (t,u_0)(x)\longrightarrow 0$ as $t\to\infty$ pointwise for each $x\in[0,\pi]\,$. 
  We now prove that convergence takes place in the topology of $\operatorname{H}^1\,$. 
  To that end use \eqref{vcf} to derive the equation satisfied by 
  $\hat u_n(t)$, which is the $n$-th coefficient in the expansion of the solution 
  $$
  u(t,x)=\sum_{k=1}^{\infty} \langle u(t,\cdot),\varphi_k \rangle \varphi_{k}(x)=
  \sum_{k=1}^{\infty} \hat u_k(t) \varphi_{k}(x),.
  $$
Observe that the $\operatorname{H}^1$ norm of a function $u$ as above is 
equivalent to
$$
 \| u \| _{\operatorname{H}^1}^2=\sum _{k=0}^\infty (1+k^2)|\hat u_k|^2\,.
$$
This is seen by extending $u$ to a periodic function $\tilde u$ by 
reflection, i.e. by setting $\tilde u(x)=u(2\pi-x)$ for $x\in[\pi,2\pi]$, and 
noticing the direct relation between the standard Fourier series of 
$\tilde u$ and the spectral basis expansion of $u\,$. We also use the fact that 
$u\in\operatorname{H}^1(0,\pi)$ if and only if $\tilde u\in\operatorname{H}^1_{per}(0,2\pi)\,$.  
For $n=0$ we have that
$$
 \hat u_0(t)=\bar u(t)=\bar u_0-\frac{1}{\pi}\int_0^t g_\beta 
 \bigl( u(\tau,\pi)\bigr)\, d\tau, 
$$
which converges to zero as $t$ approaches infinity as we have seen above.
For $n\geq 1$ one has that
$$
 \hat u_n(t)=e^{-tn^2}\hat u_{0n}-\sqrt{\frac{2}{\pi}}\int_0^t g_\beta \bigl(
 u(\tau,\pi)\bigr)e^{-(t-\tau)n^2}\, d\tau\:. 
$$
A simple calculation using the boundedness of $g_\beta$ then yields
$$
 (1+n^2)\big | \hat u_n(t)\big |^2\leq c(1+n^2) |\hat u_{0n}|+
 \frac{c}{n^4}(1+n^2),\: n\geq 1.
$$
This, together with the fact that $u_0\in\operatorname{H}^1$, implies 
that the series 
\begin{equation}\label{seriesh1}
    \sum_{n\geq 1}(1+n^2)|\hat u_n(t)|^2
\end{equation}
converges uniformly in $t\geq 0$. Arguing as in the first part of
the proof, i.e. by using the integral representation $\eqref{vcf}$ 
and splitting it, we obtain 
$$
 (1+n^2)|\hat u_n(t)|^2\to 0\text{ as }t\to\infty,
$$
for any $n\geq 0$ and we can infer that $u(t)\to 0$ in
$\operatorname{H}^1$ by combining this with the uniform
convergence of the series \eqref{seriesh1}. Note that the tail of the series can be made small 
uniformly in time, while the remaining finite sum can be estimated by the choice of 
a sufficiently large time.
\end{proof}
\section{Relationship to Feedback Control Problems}
No reference will be made to this section in the rigorous development
of the proof for the stability result. Nevertheless we include it before 
proceeding to the analysis of the asymptotic behaviour of the Volterra 
integral equation derived in the previous section. 
Our concern is that, otherwise, the proof may appear as a rather 
arbitrary succession of technical results for a Volterra integral 
equation, where, in the end, somewhat miraculously, the desired result 
emerges out of the blue. In addition the linear and streamlined 
exposition that results when omitting this section would not at all reflect 
the rather convoluted path which eventually led to this proof.  
Readers familiar with feedback control systems and distributed parameter systems, 
in particular those with knowledge of the celebrated Nyquist and Popov 
criteria may omit this section without fearing any gaps in rigour for the sequel of 
this study. We refer to \cite{K02}, \cite{CM09} for more information on 
feedback control systems.            

A classical $n\in \mathbb{N}$-dimensional single-input-single-output
(SISO) feedback system, which is also a special case of Lur'e system
(which allows for multiple inputs and outputs), takes the form
$$\begin{cases}
  \dot x= Ax+bu&\\
  y=c^\top x&\\
  u=h(y)
\end{cases}$$
for a matrix $A\in \mathbb{R} ^{n\times n}$, for vectors $b,c\in \mathbb{R}
^{n\times 1}$, and for a nonlinear function $h:\mathbb{R}\to
\mathbb{R}$. An important object in the study of such systems is
the so-called {\em transfer function} $G$ given by
$$
 G(s)=(s-A)^{-1}\mathcal{L}\bigl( h\circ y\bigr).
$$
It is obtained from the system by taking a Laplace transform $\mathcal{L}$ 
with zero initial condition. Assuming $h(0)=0$ yields the equilibrium 
$y\equiv 0$ and one is interested in its stability. This has attracted a 
great deal of interest, especially for the class of nonlinearities
satisfying a so-called {\em local [global] sector condition},
i.e. such that
$$
 \alpha z^2\leq zh(z)\leq \beta z^2\text{ for }z\in (a,b)\: [z\in
 \mathbb{R}]
$$
for real numbers $a<b$. Two stability conditions which emerged from the
research go by the name of the {\em circle} and the {\em Popov}
criteria, respectively. They can be formulated in terms of the
transfer function of the system. The present paper has drawn inspiration from
these developments of control theory in view of the possibility of
thinking of \eqref{tsACP} as an infinite dimensional SISO feedback
system. The correspondence is given by
$$
 x\to u,\: A\to -A_N,\: b\to\delta_0,\: c^\top\to\gamma_\pi\:(=\delta_\pi^\top),
 \text{ and }h(y)\to\tanh(\beta y),
$$
where $h$ satisfies a global sector condition with $\alpha=0$ and
$\beta>0$.\\

The above control system is a closed loop feedback system. In feedback 
control a successful strategy consists in deducing the properties of a 
closed loop system from the analysis of the corresponding open loop one 
obtained by simply cutting the feedback loop. 
The obtained open loop system is then just an input-ouput black box 
without feedback. When the output of the black box depends linearly on 
the applied input a so-called open loop frequency scan across all pure
sinusoidal input frequencies determines the amplitude and phase response 
of the output of the open loop system. The 2D open loop frequency response
diagram where output amplitude and output phase are plotted across all
frequencies is a workhorse in electrical and mechanical engineering. 
The method is known as the Bode or Nyquist stability plot. The importance 
of the Nyquist plot lies in the fact that the maximal feedback amplification
parameter still preserving asymptotic stability of the rest state, the 
so-called maximal gain, can be determined from the transfer function of 
the open loop system. Often it is simply read off from a Nyquist plot 
determined by empirical frequency response measurements.  
The Popov criterion applies to a wide class of nonlinear feedback systems. 
It combines the frequency response of the linear open loop system via its
transfer function with the sectorial parameters of the nonlinear feedback
function. One single combined criterion formulated in terms of the location 
of a certain curve in the complex plane (the Popov plot) provides a 
flexible and often optimal stability analysis tool.   
\subsection{The nonlocal linear feedback problem}
To illustrate this approach we first study the following linear open loop 
system related to problem \eqref{tsEq}
\begin{equation}\label{OLSnonl}
  \begin{cases}
    u_t-u_{xx}=0 &\text{in }(0,\infty)\times(0,\pi),\\
    u_x(t,0)=-f(t) &\text{for }t\in(0,\infty),\\
    u_x(t,\pi)=0 &\text{for }t\in(0,\infty),\\
    u(0,\cdot)\equiv 0 &\text{in }(0,\pi).
  \end{cases}
\end{equation}
Here $f(t)\in \operatorname{BC}(\mathbb{R}^+,\mathbb{R})$ is interpreted as 
the input to the system at $x=0$. Note that we choose the equilibrium state
$u_{0}\equiv 0$ as the initial condition.
The output is obtained by measuring the temperature at $x=\pi$ by taking the
trace of the solution $u(t,u_0)$ of the initial value \eqref{OLSnonl} which is
given by 
$$
  u(t,u_0) =e^{-tA_N}u_0 + \int _0^t e^{-(t-\tau)A_N} \gamma_0' 
  \bigl( f(\tau) \bigr) \, d\tau =\int _0^t f(\tau) e^{-(t-\tau)A_N} 
  \delta_0\, d\tau
$$
Hence the system output $g(t)$ is given by

\begin{equation}\label{OLSnonlout}
g(t):=\gamma_{\pi}\bigl( u(t,u_0) \bigr) = \int _0^t f(\tau) \bigl[  e^{-(t-\tau)A_N} \delta_0\bigr](\pi)  \, d\tau \,.
\end{equation}
In the previous section we discussed that 
$$
a(t)=-\bigl[  e^{-tA_N} \delta_0\bigr](\pi)=-\frac{1}{\pi}+\frac{2}{\pi}\sum_{k=1}^\infty(-1)^{k+1} e^{-tk^2}  
$$
and therefore the output is computed by convolution of the input with
the kernel $-a$
\begin{equation}\label{conv1}
g(t)=-\int _0^t a(t-\tau)f(\tau) \, d\tau \,= \,(-a*f)(t).
\end{equation}
We can now apply the Laplace transform to both sides of \eqref{conv1} to obtain
\begin{equation}\label{Transf1}
\mathcal{L}(g)=\mathcal{L}(-a)\mathcal{L}(f).
\end{equation}
The transfer function is defined as $G_{\text{nloc}}:=\mathcal{L}(-a)$ 
and was already computed in the previous section
\begin{equation}\label{Transf2}
G_{\text{nloc}}(s)=
\mathcal{L}\Bigl( \frac{1}{\pi}+\frac{2}{\pi}\sum_{k=1}^\infty(-1)^{k}
e^{-tk^2}\Bigr)(s)=\frac{1}{\pi s}+\frac{2}{\pi}\sum _{k=1}^\infty
\frac{(-1)^k}{s+k^2}= \frac{1}{\sqrt{s}\,\sinh(\pi\sqrt{s})}.
\end{equation}
The closed loop feedback system with gain $\beta$ corresponding to
\eqref{OLSnonl} is given by
\begin{equation}\label{CLSnonl}
  \begin{cases}
    u_t-u_{xx}=0&\text{in }(0,\infty)\times(0,\pi),\\
    u_x(t,0)=\beta u(t,\pi) &\text{for }t\in(0,\infty),\\
    u_x(t,\pi)=0&\text{for }t\in(0,\infty),\\
    u(0,\cdot)\equiv 0&\text{in }(0,\pi),
  \end{cases}
\end{equation}
In feedback control analysis, the closed loop system transfer function is
obtained from the open loop transfer function. Adopting the usual definition
from finite dimensional feeback systems heuristically, we set
$$
 G_{\text{nloc}}^{\text{cl-loop}}(\cdot,\beta)=
 \frac{G_{\text{nloc}}}{1+\beta\,G_{\text{nloc}}}.
$$
It is clear that the zeros of the closed loop transfer function and the 
zeros of the open loop transfer function coincide. 
The poles of the closed loop transfer function correspond to the set 
$$
 \big\{s\in  \mathbb{C} \big | \beta\,G_{\text{nloc}}(s) = -1\big\}.
$$
This leads to the use of Rouch\'e's theorem to derive the Nyquist stability
criterion. It gives sufficient conditions for the stability of the trivial
equilibrium in terms of a vanishing winding number of the Nyquist curve 
$$
 \text{Nyquist}\bigl( \beta\,G_{\text{nloc}} \bigr):=\big\{ 
 \beta\,G_{\text{nloc}}(i\omega) \,\big |\, \omega\in 
 \mathbb{R}\big\}\,,
$$
around the point $(-1,0)$ in the complex plane. In Figure
\ref{fignyqnonloc} the Nyquist plot is computed numerically
for $\beta=1$ from the representation
\begin{equation}\label{explicitrep2}
 G_{\text{nloc}}(i\omega)=\frac{1}{\sqrt{i\omega}\,
 \sinh(\pi\sqrt{i\omega})}
\end{equation}
for numerical values in $\omega\in (-50,-0.2)\cup (0.2,50)\,$. 
A semi-circle in the complex right half-plane of radius $0.2$ is followed to avoid the singularity at zero. 
This leads to a closed Nyquist curve. 
\begin{figure}
\centering
        \includegraphics[totalheight=8cm]{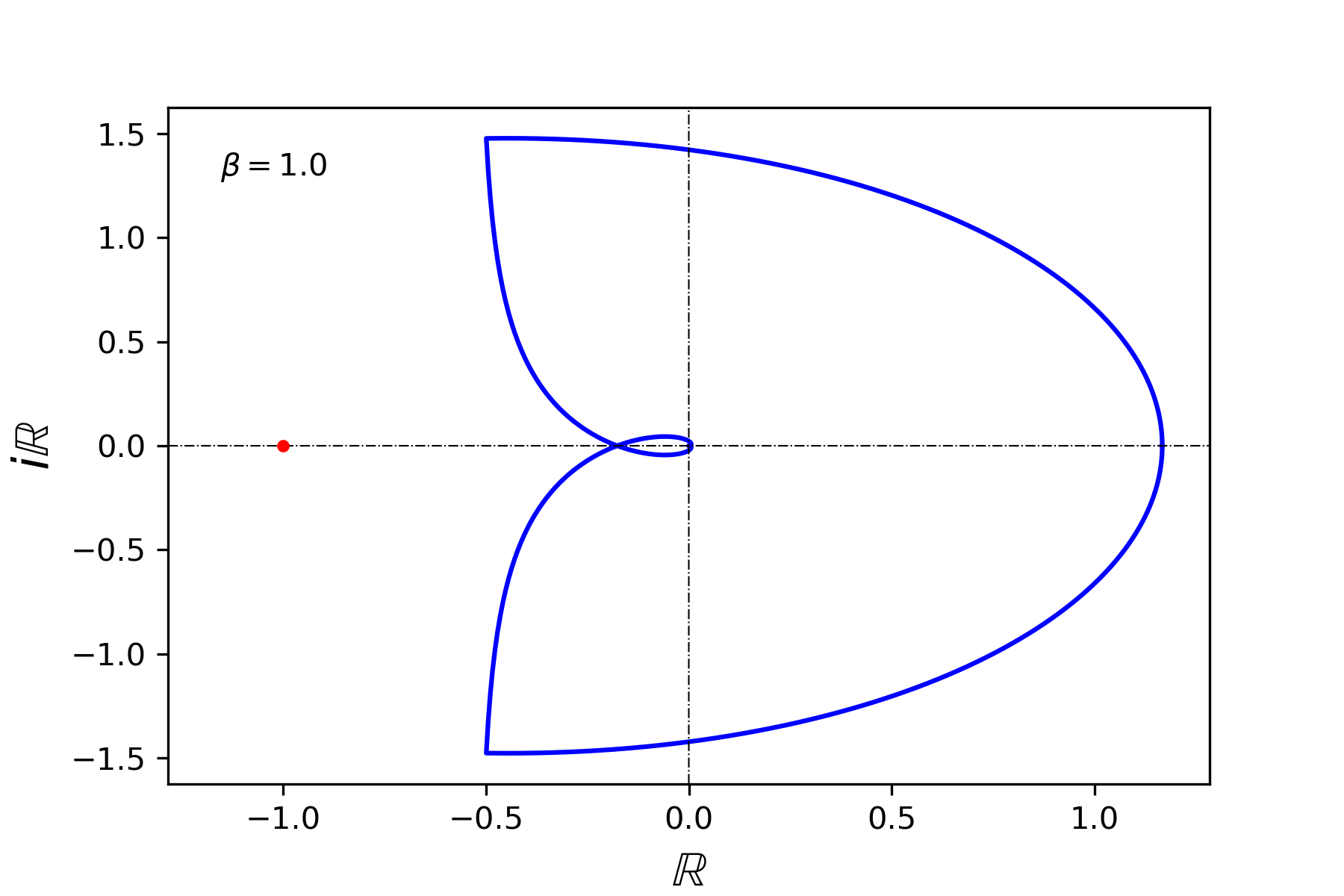}
    \caption{Nyquist plot $\text{Nyquist}\bigl(\,G_{\text{nloc}} \bigr)\,$. The limited resolution 
    hides the infinitely many windings of the curve around the origin as $\omega\to\pm\infty\,$. }
    \label{fignyqnonloc}
\end{figure}
A solution $(\omega_{0},\beta_{0})$ such that $G_{\text{nloc}}(i\omega_{0})
\in (-\infty,0)$ and 
$$
\beta_{0}\,G_{\text{nloc}}(i\omega_{0})=-1  
$$
leads to a critical parameter value $\beta_0\,$. It can be found from the condition 
$$
 \operatorname{Im}\bigl[ G_{\text{nloc}}(i\omega_0)\bigr]=
 -\frac{1}{\pi \omega_0}-\frac{2}{\pi}\sum _{k=1}^\infty (-1)^k
 \frac{\omega_0}{\omega_0^2+k^4}=0
$$
or, equivalently, from
\begin{equation}\label{NyCond}
 \sum _{k=1}^\infty \frac{(-1)^{k-1}}{1+k^4/\omega_0^2}
 =\frac{1}{2}\;.
\end{equation}
It appears that \eqref{NyCond} possesses infintely many real solutions. 
In fact, by using the explicit representation \eqref{explicitrep2} a calculation that 
separates real and imaginary parts of the denumerator in \eqref{explicitrep2} and 
uses the expression for the imaginary part of $\frac{1}{a+ib}$ leads to an 
explicit equivalent expression for the condition \eqref{NyCond}, given by 
$$
\sinh(\pi\sqrt{\frac{\omega}{2}})\cos(\pi\sqrt{\frac{\omega}{2}})+
\cosh(\pi\sqrt{\frac{\omega}{2}})\sin(\pi\sqrt{\frac{\omega}{2}})=0\:.
$$
Clearly this is equivalent to the condition already found in \cite{GM97} (Remarks 4.2.)
$$
\tan(\pi\sqrt{\frac{\omega}{2}})=-\tanh(\pi\sqrt{\frac{\omega}{2}})
$$
and shows that $\eqref{NyCond}$ has infinitely many solutions.\\ 
Picking the smallest solution, which occurs for $\omega_0 \approx 1.13344388$, 
it is found by inserting, that $G_{\text{nloc}}(i\omega_0)\approx -0.17650842\,$. 
The critical value for $\beta_0$ is finally obtained from 
$$
 \beta_{0}=-\frac{1}{G_{\text{nloc}}(i\omega_0)}=-\sqrt{i\omega_0}\,
 \sinh(\pi\sqrt{i\omega_0})\approx 5.6655. 
$$
Hence this formal application of the Nyquist criterion suggests that the 
zero state of \eqref{CLSnonl} is asymptotically stable for 
$\beta \in (0,\beta_0)$. 
This stability interval determined heuristically is identical to the 
(linear) stability interval determined rigorously in \cite{GM97} by fully
describing the spectrum of the generator of the linear semigroup
associated with \eqref{CLSnonl}.


\subsection{The local linear feedback problem}
We also briefly discuss the local version of \eqref{OLSnonl} and also refer to \cite{CM09} where a similar stability analysis of this problem can be found. 
We start again with an input signal $f(t)$ that enters as before:  
\begin{equation}\label{OLSloc}
  \begin{cases}
    u_t-u_{xx}=0 &\text{in }(0,\infty)\times(0,\pi),\\
    u_x(t,0)=-f(t) &\text{for }t\in(0,\infty),\\
    u_x(t,\pi)=0 &\text{for }t\in(0,\infty),\\
    u(0,\cdot)\equiv 0 &\text{in }(0,\pi).
  \end{cases}
\end{equation}
Now the output is obtained by measuring the temperature at $x=0$ by taking the
trace $\gamma_0$ of the solution $u(t,u_0)$ of the initial value \eqref{OLSloc}. 
Hence the system output $g(t)$ is given by
\begin{equation}\label{OLSlocout}
g(t):=\gamma_{0}\bigl( u(t,u_0) \bigr) = \int _0^t f(\tau) \bigl[  e^{-(t-\tau)A_N} \delta_0\bigr](x=0)  \, d\tau \,.
\end{equation}
For this local version we now determine the kernel by evaluating at $x=0$  
$$
a_{\text{loc}}(t):=-\bigl[  e^{-tA_N} \delta_0\bigr](x=0)=-\frac{1}{\pi}-\frac{2}{\pi}\sum_{k=1}^\infty\, e^{-tk^2} \,. 
$$
Again the output is given by convoluting the input with the kernel $a_{\text{loc}}$
\begin{equation}\label{conv1}
g(t)=-\int _0^t a_{\text{loc}}(t-\tau)f(\tau) \, d\tau \,= \,(-a_{\text{loc}}*f)(t)
\end{equation}
and by applying the Laplace transform
\begin{equation}
\mathcal{L}(g)=\mathcal{L}(-a_{\text{loc}})\mathcal{L}(f).
\end{equation}
The transfer function is defined as $G_{\text{loc}}:=\mathcal{L}(-a_{\text{loc}})$ and 
it can be expressed as
\begin{equation}\label{Transf2}
G_{\text{loc}}(s)=
\mathcal{L}\Bigl( \frac{1}{\pi}+\frac{2}{\pi}\sum_{k=1}^\infty
e^{-tk^2}\Bigr)(s)=\frac{1}{\pi s}+\frac{2}{\pi}\sum _{k=1}^\infty
\frac{1}{s+k^2}= \frac{\cosh(\pi\sqrt{s}))}{\sqrt{s}\,\sinh(\pi\sqrt{s})}.
\end{equation}
The local closed loop feedback system with gain $\beta$ is now given by
\begin{equation}
  \begin{cases}
    u_t-u_{xx}=0&\text{in }(0,\infty)\times(0,\pi),\\
    u_x(t,0)=\beta u(t,0) &\text{for }t\in(0,\infty),\\
    u_x(t,\pi)=0&\text{for }t\in(0,\infty),\\
    u(0,\cdot)\equiv 0&\text{in }(0,\pi).
  \end{cases}
\end{equation}
In Figure \ref{fignyqloc} we compute the Nyquist plot from the representation \eqref{Transf2} and observe that the Nyquist plot lies in the right complex half-plane. 
Therefore the stability range for the local problem obtained from the above heuristic application of the Nyquist
criterion is $\beta\in(0,\infty)\,$. We point out that this already follows rigorously from the analysis of the spectrum of the generator of the linear semigroup, 
as is discussed in \cite{GM97} and \cite{GM99}. In fact, for a wide class of 
nonlinear parabolic evolution problems in one space dimension with ``local'' or ``separated'' boundary conditions, the
general results of P. Pol\'a\v{c}ik et. al. (see e.g. \cite{BrPolSan92}) show that
all orbits converge to steady states, whenever a compact global attractor is known to  exist.
\begin{figure}
\centering
    \includegraphics[totalheight=8cm]{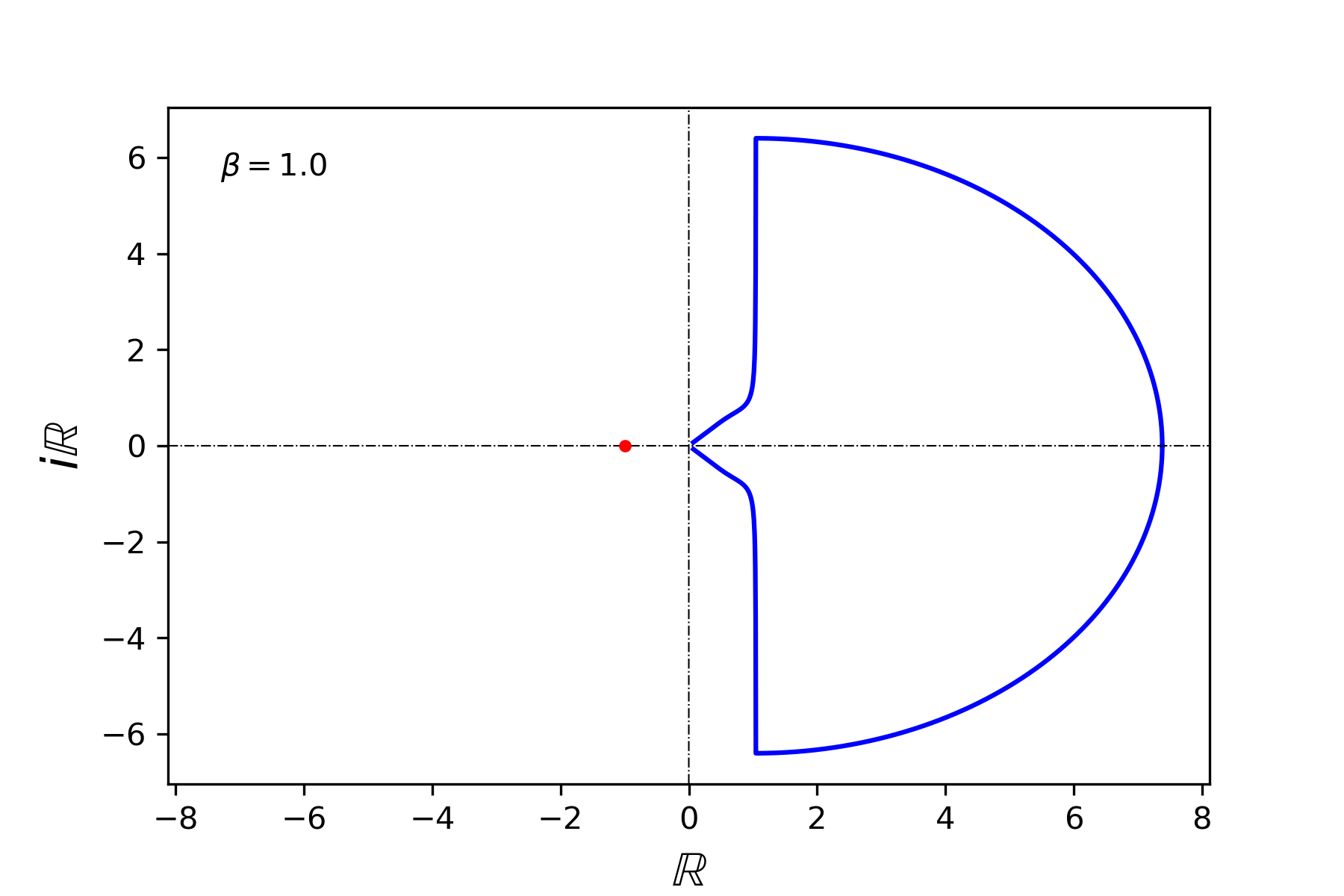}
    \caption{Nyquist plot of the transfer function $G_{\text{loc}}\,$.}
    \label{fignyqloc}
\end{figure}
\subsection{Nyquist criterion and spectral linear stability analysis}\label{NyqLinStab}
This short section aims at drawing a formal connection (with no proofs) between the
mathematical linear stability analysis originally performed in
\cite{GM97} and the Nyquist diagram approach found in
applications and just described. It is known from semigroup theory
\cite{Gol85} that the Laplace transform of a strongly continuous
semigroup $\big\{ T_A(t)\, |\, t\geq 0\big\}$ on a Banach space $E$ is
the resolvent of its generator, i.e. that 
$$
 \mathcal{L}\bigl( T_A(\cdot)\bigr)(s)=(s-A)^{-1}\text{ for }s\in \rho(A),
$$
with $s\geq \omega(A)$, where $\omega(A)$ is the smallest real number
for which $\|T_A(t)\|_{\mathcal{L}(E)}\leq M\, e^{\omega t}$ for
$t\geq 0$ and some constant $M\geq 1$. This formula hints at the fact
that the poles of the resolvent are the modes of decay/growth of
$T_A$, which, in turn, correspond to the eigenvalues of $A$. For the
semigroup $e^{-\cdot A_N}$ the eigenvalues $\lambda\in \mathbb{C}$ are
determined by the equation
$\sqrt{\lambda}\sin(\pi\sqrt{\lambda})=\beta$ as follows from
\cite{GM97}, while $\bar u=\mathcal{L}(u)$ satisfies the identity
\begin{equation}\label{ubar}
 \bar u(s,x)=\bigl[ (s+A_N)^{-1}u_0\bigr] (x)-\beta
 \bar u(s,\pi) \bigl[ (s+A_N)^{-1}\delta _0\bigr] (x)
\end{equation}
as follows from the weak formulation 
$$
 (s+A_N)\bar u=u_0-\beta \bar u(\pi)\delta_0\text{ in}\operatorname{H}^{-1}, \bar u\in \operatorname{H}^1,
$$
of the Laplace transform of \eqref{tsEq} given by
\begin{equation*}\begin{cases}
    s \bar u-u_0-\bar u_{xx}=0&\text{ in }(0,\pi),\\
    \bar u_x(0)=\beta \bar u(\pi),&\\
    \bar u_x(\pi)=0,&
\end{cases}
\end{equation*}
which confirms that
$$
 \mathcal{L}\bigl( e^{-t A_N}\delta_0\bigr)(s) = \bigl[(s+A_N)^{-1}\delta _0\bigr](\pi)=\frac{1}{\sqrt{s}\sinh(\pi \sqrt{s})}
$$
is the transfer function of the open loop control system. Equation
\eqref{ubar} allows one to recover $\bar u(s,\pi)$ by simple
evaluation at $x=\pi$ to get
$$
 \bar u(s,\cdot)=(s+A_N)^{-1}u_0-\beta
 \frac{\bigl[(s+A_N)^{-1}u_0\bigr](\pi)}{1+\beta\bigl[(s+A_N)^{-1}\delta
 _0\bigr](\pi)}(s+A_N)^{-1}\delta _0,
$$
an explicit representation of the solution of the closed loop
system. Notice that, after evaluation at $x=\pi$, equation
\eqref{ubar} is simply the Laplace transform of the
Volterra integral equation \eqref{vie} in the linear case. Now the
modes of decay (or growth) of $u$ are fully determined by the poles of
$\bar u$, which, in turn, are determined by the zeros of $1+\beta
\bigl[(s+A_N)^{-1}\delta _0\bigr](\pi)$, since these are the only poles of $\bar u$
which change with $\beta$. Since the Nyquist plot is a complex analytical
procedure to identify poles in the ``bad'' half-plane, it ends up
being a method to check when stability is lost for the linear problem
and, hence, amounts to a linearized stability criterion for the
original nonlinear system, in this case.


\subsection{The Popov criterion applied to the nonlocal and nonlinear
  feedback problem} 
The Popov criterion for a nonlinearity satisfying the global sector
condition $[0,\beta]$ is stated in terms of the transfer function of
the linear open loop system.
Applied to our transfer function the Popov criterion requires that,
for a given parameter $\beta>0$, there exists a constant $q(\beta)>0$
such that, for $\omega\in\mathbb{R}$,
\begin{equation}\label{popcrit1}
\operatorname{Re}\bigl[ G_{\text{nloc}}(i\omega)\bigr] - 
q\omega\operatorname{Im}\bigl[ G_{\text{nloc}}(i\omega)\bigr] > -\frac{1}{\beta}\;.    
\end{equation}
 If we define the Popov plot by 
\begin{equation}\label{popcrit2}
  P(G_{\text{nloc}}):= \Big\{\operatorname{Re}\bigl[
  G_{\text{nloc}}(i\omega) \bigr] + i \omega \operatorname{Im}
  \bigl[ G_{\text{nloc}}(i\omega) \bigr] \,\Bigl |\, \omega > 0\Big\}\,
\end{equation}
then \eqref{popcrit1} is equivalent to the existence of a constant
$q>0$ such that the straight line in the complex plane with slope $q$
through the real point $-\frac{1}{\beta}$ lies to the left of the
Popov plot $P(G_{\text{nloc}})$. If the criterion is fulfilled, then
any nonlinearity in the sector $[0,\beta]$ leads to an asymptotically
stable nonlinear feedback system. Note that our particular
nonlinearity $\tanh(\beta \cdot)$ satifies the sector condition
$[0,\beta]$ and therefore the nonlinear evolution problem
\eqref{tsEq}, at least formally, falls into the class of nonlinear
feedback systems covered by the Popov criterion.

The Popov plot \eqref{popcrit2} can easily be computed numerically
from the explicit representation \eqref{Transf2} and is displayed in
Figure \ref{figpopov}. Since the Popov plot crosses the real line at
$-\frac{1}{\beta_{0}}$ where $\beta_{0}\approx 5.6655$ is the constant
already discussed above, the criterion suggests that the parameter
range of stability is given by $(0,\beta_{0})\,$.
In the next section we will prove this result by studying the
nonlinear Volterra integral equation \eqref{vie}.\\
In the course of our rigorous analysis, which will be resumed in the
next section, the application of the Parseval-Plancherel identity will
lead to a Fourier representation of a stability condition in terms of
the kernel of the integral equation. The stability condition obtained
in this way then turns out to be equivalent to the Popov criterion and
allows for a rigorous proof of the above heuristic argument via the
integral equation.
\begin{figure}[H]
\centering
        \includegraphics[totalheight=8cm]{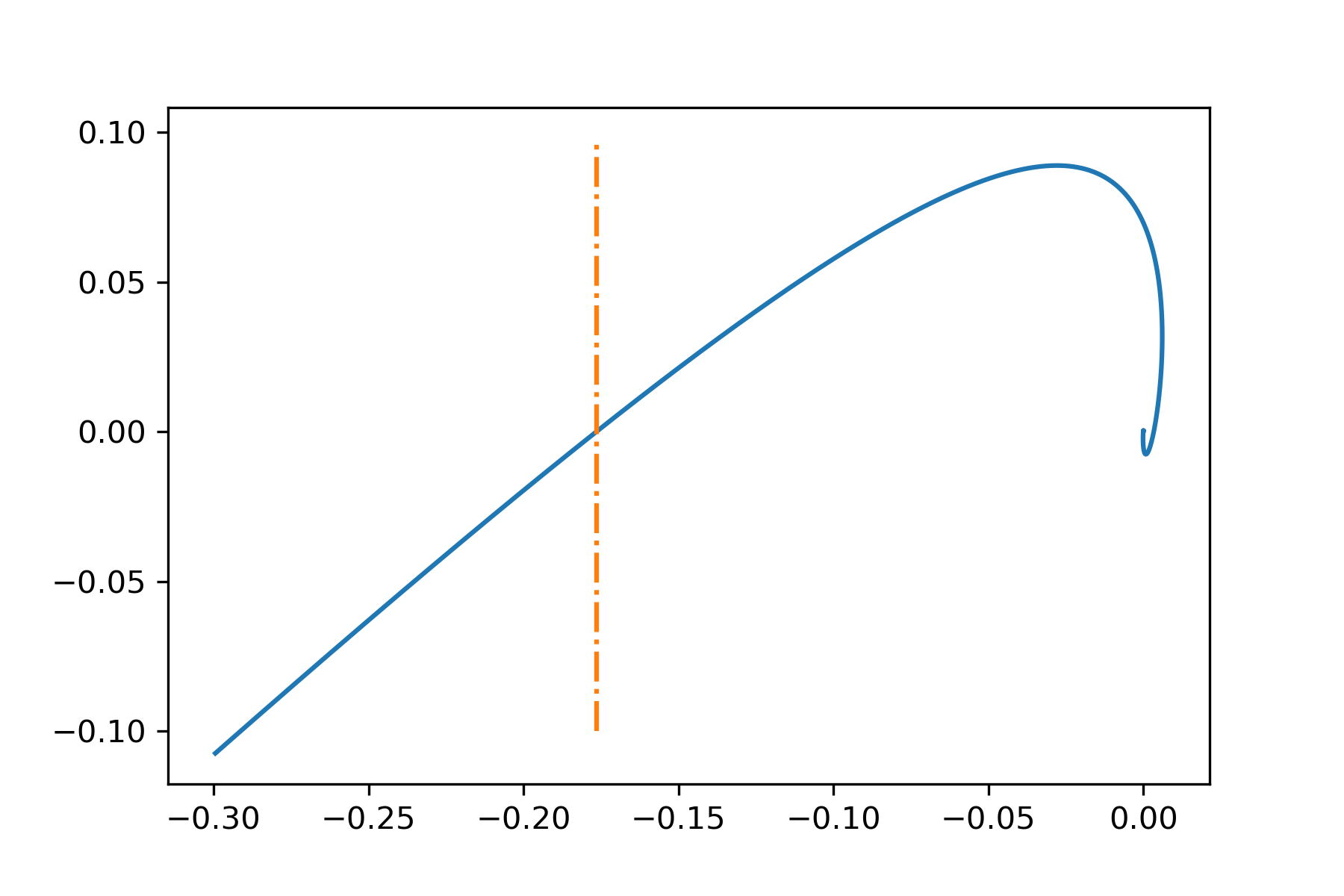}
    \caption{Popov plot $P\bigl(\,G_{\text{nloc}} \bigr)\,$. The limited resolution 
    hides the infinitely many windings of the curve around the origin as $\omega\to\infty\,$. }
    \label{figpopov}
\end{figure}
\section{Asymptotic Stability for the Integral Equation}
The objective of this section is the proof of the following
result. The ideas originate in \cite{Mi71} and are adapted to obtain the
following proposition. In particular, we have adapted the expression for 
$W_1$ in \cite{Mi71} by subtracting the $\beta$-dependent term. 
\begin{prop}\label{vieSol20}
Fix $\beta\in(0,\beta_0)$ and $u_0\in \operatorname{H}^1$. Let 
$y\in\operatorname{BC}\bigl( (0,\infty),\mathbb{R}\bigr)$ be a
solution  of the integral equation \eqref{vie}. Then
$\lim_{t\to\infty}y(t)=0$. 
\end{prop}
For the proof we need a series of lemmas and an auxiliary function.
\begin{deff}
For $\beta,q\in(0,\infty)$ and $y\in\operatorname{BC}\bigl(
[0,\infty),\mathbb{R}\bigr)$ set
$$
 W_{\beta,q}(y)(t):=\sum_{i=1}^3W_i(y)(t),\: t\geq 0,
$$
where
\begin{align*}
 W_1(y)(t)&:=\int _0^t g_\beta \bigl( y(\tau)\bigr)\Bigl[
 y(\tau)-\frac{g_\beta \bigl( y(\tau)\bigr)}{\beta}\Bigr]\, d\tau,
 \\W_2(y)(t)&:=q\,G_{\beta} \bigl( y(t)\bigr)\text{ for }G_{\beta}(z):=\int _0^z
 g_\beta(\zeta)\, d\zeta,\\
 W_3(y)(t)&:=\frac{1}{2\pi}\Bigl( \int _0^t g_\beta \bigl(
 y(\tau)\bigr)\, d\tau \Bigr)^2 .
\end{align*}
Note that in our notation we may not always explicitly indicate the dependence on $\beta$ and
$q$ in the notation for $W_i$, $i=1,2,3$ nor the dependence
on the function $y$. 
\end{deff}
\begin{lem}\label{nonneg}
Let $\beta,q\in(0,\infty)$ and $y\in\operatorname{BC}\bigl(
[0,\infty),\mathbb{R}\bigr)$. Then
$$
 W_i(t)\geq 0,\: t\geq 0,\: i=1,2,3.
$$
It therefore also holds that $W_{\beta,q}\geq 0$ for $t\geq 0$.
\end{lem}
\begin{proof}
The nonnegativity of $W_1$ follows from its nonnegative integrand. 
First note that $g_{\beta}(y)=0$ for $y=0$ and hence the integrand 
vanishes for $y=0$. For $y\neq0$ we can express the integrand as
$$
\frac{1}{\beta}\,g_{\beta}(y)\beta\, y \bigl[ 1 -\frac{g_{\beta}(y)}{\beta\,y} \bigr]
$$
and the positivity of this expression follows from $g_{\beta}(y)\beta\, y\,>0$ and 
from $\frac{g_{\beta}(y)}{\beta\,y}\in(0,1)\,$. 
Also the nonnegativity of $W_{2}$ is a consequence of 
$g_{\beta}(y)\, y\,>0$ for $y\neq0\,$. 
\end{proof}

\begin{lem}\label{VandR}
Let $\beta,q\in(0,\infty)$ and let $y\in\operatorname{BC}\bigl(
(0,\infty),\mathbb{R}\bigr)$ be a solution of the integral equation
\eqref{vie} with $u_0\in \operatorname{H}^1\,$. Then
\begin{equation}\label{Lyap}
  W_{\beta,q}(t)=V_{\beta,q}(t)+R_{\beta,q}(t),\: t\geq 0\,,
\end{equation}
where
$$
 V_{\beta,q}(t):=\int _0^t g_\beta \bigl( y(\tau)\bigr)\bigl[
 f(\tau)+qf'(\tau)\bigr]\, d\tau+qG_{\beta} \bigl( y(0)\bigr) 
$$
and
$$
 R_{\beta,q}(t):=\int _0^t g_\beta \bigl( y(\tau)\bigr)\Big\{
 \int _0^\tau \bigl[ a_s(\tau-\sigma)+q\, a_s'(\tau-\sigma) \bigr]
 g_\beta \bigl( y(\sigma)\bigr)\, d\sigma -\frac{g_\beta
 \bigl( y(\tau)\bigr)}{\beta}\Big\}\, d\tau.
$$
Using convolutions, the latter expression can be written more concisely as
$$
 R_{\beta,q}(t)=\int _0^t g_\beta \bigl(
 y(\tau)\bigr)J_{\beta,q}(\tau)\, d\tau,
$$
where $J_{\beta,q}$ is defined for $\tau\geq 0$ by
\begin{equation}\label{J}
  J_{\beta,q}:=\bigl[ a_s+q a_s'\bigr]*g_\beta \bigl(y(\cdot)\bigr)-
  \frac{g_\beta\bigl( y(\cdot)\bigr)}{\beta} \,.
\end{equation}
\end{lem}
\begin{proof}
First note that by \eqref{vie} and since $a(0)=0$ by Remark \ref{rems} we obtain
\begin{equation}\label{deriy}
y'(t)=f'(t)+ a(0)\,g_\beta \bigl( y(t) \bigr) 
+ \int _0^t a'(t-\tau)g_\beta\bigl( y(\tau)\bigr)\, d\tau 
=f'(t)+  
\int _0^t a'(t-\tau)g_\beta\bigl( y(\tau)\bigr)\, d\tau \,.
\end{equation}
Then by the simple transformation $dy=y'(\tau)d\tau$ we can rewrite $W_2$ as  
\begin{equation}\label{w2new}
G_{\beta} \bigl( y(t)\bigr) = G_{\beta} \bigl( y(0)\bigr) + \int _0^t g_\beta(y(\tau))\,y'(\tau)\, d\tau \,.
\end{equation}
By partial integration we can rewrite $W_3$ as
\begin{equation}\label{w3new}
W_3(y)(t)=\frac{1}{\pi}\int_{0}^{t} \Bigl[\, \int_{0}^{\tau} g_\beta\bigl( y(\sigma)\bigr)\, d\sigma \Bigr] g_\beta\bigl( y(\tau)\bigr)\, d\tau\,.  
\end{equation}
Now insert \eqref{deriy} in \eqref{w2new} and use the above expressions \eqref{deriy}--\eqref{w3new} in the definition of W. 
The claimed decomposition is obtained by collecting the integrals contaning $f$ and $f'$ and the constant $q \,G_{\beta}(0)$, which produces $V_{\beta,q}\,$. 
The remainining terms are found to be equal to $R_{\beta,q}$ by using the identity $a_s=a+\frac{1}{\pi}$  to write 
$$
\int _0^\tau a(\tau -\sigma) g_\beta\bigl( y(\sigma)\bigr)\, d\sigma + \frac{1}{\pi} \int _0^\tau  g_\beta\bigl( y(\sigma)\bigr)\, d\sigma = \int _0^\tau a_s\,(\tau-\sigma) g_\beta\bigl( y(\sigma)\bigr)\, d\sigma \,.
$$
and by replacing $a'$ by $a'_s$ in the other term so that finally the kernel $a_s+q\,a'_s$ results in the expression for $R_{\beta,q}\,$. 

The more concise representation of $R_{\beta,q}$ is clear, since we simply use the definition of the convolution in our specific situation  
$$
\Bigl( \bigl[ a_s+q a_s'\bigr]*g_\beta \bigl(y(\cdot)\bigr) \Bigr)\,(\tau) = 
\int _0^\tau \bigl[ a_s(\tau-\sigma)+q\, a_s'(\tau-\sigma) \bigr]
 g_\beta \bigl( y(\sigma)\bigr)\, d\sigma\,. 
$$
\end{proof}
In the next lemma we apply the Parseval-Plancherel Theorem to derive a Fourier representation for $R_{\beta,q}\,$. 
This will allow us to obtain the criterion that makes $R_{\beta,q}\,<0$ so that the nonnegative quantity $W_{\beta,q}$ will be bounded by $V_{\beta,q}$ from above.
\begin{lem}
Let $\beta,q\in(0,\infty)$ and let $y\in\operatorname{BC}\bigl(
(0,\infty),\mathbb{R}\bigr)$ be a solution of the integral equation
\eqref{vie} with $u_0\in \operatorname{H}^1$. Then
$$
 R_{\beta,q}(t)=\frac{1}{2\pi}\int _{-\infty}^\infty {\widehat g_{\beta,\theta_t}}^{\,2}(\omega)\bigl[
 \hat a_s(\omega)+q\,\widehat{a_s'}(\omega)-\frac{1}{\beta}\bigr]
 \, d \omega,\: t\geq 0,
$$
where for $\tau\in \mathbb{R}$ and $t\geq 0$
$$
 g_{\beta,\theta_t}(\tau):=g_\beta \bigl( y(\tau)\bigr)\theta_t(\tau)\,,
$$
with 
$$
 \theta _t(\tau):=\begin{cases} 1,&\tau\in [0,t],\\ 0,&\tau\in
   \mathbb{R}\setminus[0,t]\, \end{cases}
$$
and $y(\tau):=0$ for $\tau<0\,$.
\end{lem}
The Fourier transforms of $a_s$ and $a_s'$ were discussed in Remarks
\ref{rems}. Since $g_{\beta,\theta_t}\in
\operatorname{L}^1(\mathbb{R})\cap \operatorname{L}^2(\mathbb{R})$ for each 
$t\geq0\,$, its Fourier transform  is defined classically.
\begin{proof}
We will use the Parseval-Plancherel identity
$$
 (f,g)_2=\int\limits_{\mathbb{R}} f(x)g(x)\, dx=
 \frac{1}{2\pi}\int\limits_{\mathbb{R}} \hat f(\omega)\hat g(\omega)\, d\omega=\frac{1}{2\pi}(\hat f,\hat g)_2,
$$
valid for $f,g\in \operatorname{L}^2(\mathbb{R})$. By setting 
$$
 J^t_{\beta,q}:=\bigl[ a_s+q\, a_s'\bigr]*g_{\beta,\theta_t}-
 \frac{g_{\beta,\theta_t}}{\beta}.
$$
we extend $J_{\beta,q}$ defined in \ref{J} on $[0,\infty)$ to $\mathbb{R}\,$. Then $J^t_{\beta,q}\in
\operatorname{L}^1(\mathbb{R})\cap \operatorname{L}^2(\mathbb{R})$ for each 
$t\geq0\,$ and the Parseval-Plancherel identity yields 
\begin{equation*}
 R_{\beta,q}(t)=\Bigl( g_{\beta,\theta_t},J^t_{\beta,q}\Bigr) _2=\frac{1}{2\pi}
 \Bigl( \widehat{g_{\beta,\theta_t}},\widehat{J^t_{\beta,q}}\Bigr) _2=\frac{1}{2\pi}
 \Bigl( \widehat{g_{\beta,\theta_t}}^2,\hat a_s+q\,
 \widehat{a_s'}-\frac{1}{\beta}\Bigr) _2\:\:.
\end{equation*}
\end{proof}
\begin{rem}
For $\beta,q\in (0,\infty)$ and a solution $y\in \operatorname{BC}\bigl(
(0,\infty),\mathbb{R}\bigr)$ of the integral equation \eqref{vie} with $u_0\in
\operatorname{H}^1$, the condition
$$
 R_{\beta,q}(t)< 0\text{ for }t\geq 0,
$$
is satisfied if
$$
 \hat a_s(\omega)+q\, \widehat{a_s'}(\omega)-\frac{1}{\beta}<0\text{ for }
 \omega\in \mathbb{R}.
$$
By using the series representation \eqref{aRep} one sees that this is equivalent to
\eqref{M} and amounts to
\begin{equation}\label{MNew}
  \frac{2}{\pi}\sum _{k=1}^\infty (-1)^{k+1} \Bigl[
  \frac{k^2+q\, \omega^2}{k^4+\omega^2}\Bigr]-\frac{q}{\pi}-\frac{1}{\beta}<0
\end{equation}
for $\omega\in \mathbb{R}$. Recalling the definition of $G_+$ given in
Remarks \ref{rems}(h), \ref{rems}(i) and the Popov criterion
\eqref{P}, which rewrites as
$$
 q\,\omega \operatorname{Im}\Bigl( G_+(i
 \omega)\Bigr)-\operatorname{Re}\Bigl( G_+(i
 \omega)\Bigr)-\frac{1}{\beta}<0, 
$$
and using the series representation we arrive at
$$
-\frac{q}{\pi}-\frac{2}{\pi}\sum_{k=1}^\infty (-1)^k\frac{q\,
  \omega^2+k^2}{k^4+\omega^2}- \frac{1}{\beta}<0,
$$
which is equivalent to \eqref{M}. Notice that, by symmetry, it is
enough to verify the criterion for $\omega\geq 0$, and since $G_+$ has
the explicit representation of Remarks \ref{rems}(h) and
\ref{rems}(i), it is possible to use it to verify the validity of the
condition.
\end{rem}
\begin{lem}\label{negLem}
Let $y\in\operatorname{BC}\bigl((0,\infty),\mathbb{R}\bigr)$ be a
solution of the integral equation \eqref{vie} with $u_0\in\operatorname{H}^1$.
Then, for each $\beta\in(0,\beta_0)$, there is $q=q(\beta)>0$ such
that
$$
 R_{\beta,q(\beta)}(t)\leq 0\text{ for }t\geq 0,
$$
where $\beta_0\approx 5.6655$ is the Hopf bifurcation value found in
\cite{GM97}. 
\end{lem}
\begin{proof}
Define the functions
$$
 A(\omega)=\frac{2}{\pi}\sum_{k=1}^\infty(-1)^{k+1}
 \frac{k^2}{k^4+\omega^2}\text{ and }B(\omega)=-
 \frac{1}{\pi}+\frac{2}{\pi}\sum_{k=1}^\infty(-1)^{k+1}
 \frac{\omega^2}{k^4+\omega^2},
$$
so that condition \eqref{MNew} is satisfied if $\sup _{\omega\in
  \mathbb{R}}\bigl[ A(\omega)+qB(\omega)\bigr]<\frac{1}{\beta}$ for
some $q>0$. This is clearly only possible as long as
$$
 \beta<\frac{1}{\inf _{q>0}\sup _{\omega\in
 \mathbb{R}}\bigl[ A(\omega)+qB(\omega)\bigr]}=\sup_{q>0}
 \frac{1}{\sup _{\omega\in\mathbb{R}}\bigl[
   A(\omega)+qB(\omega)\bigr]}=:\sup_{q>0}M(q)\, .
$$
A plot of the (numerically computed) function $M$ shows that
the critical value indeed coincides with $\beta_0$ from \cite{GM97}.
\end{proof}
%
%
\begin{figure}[H]
\begin{center}
\includegraphics[totalheight=8cm]{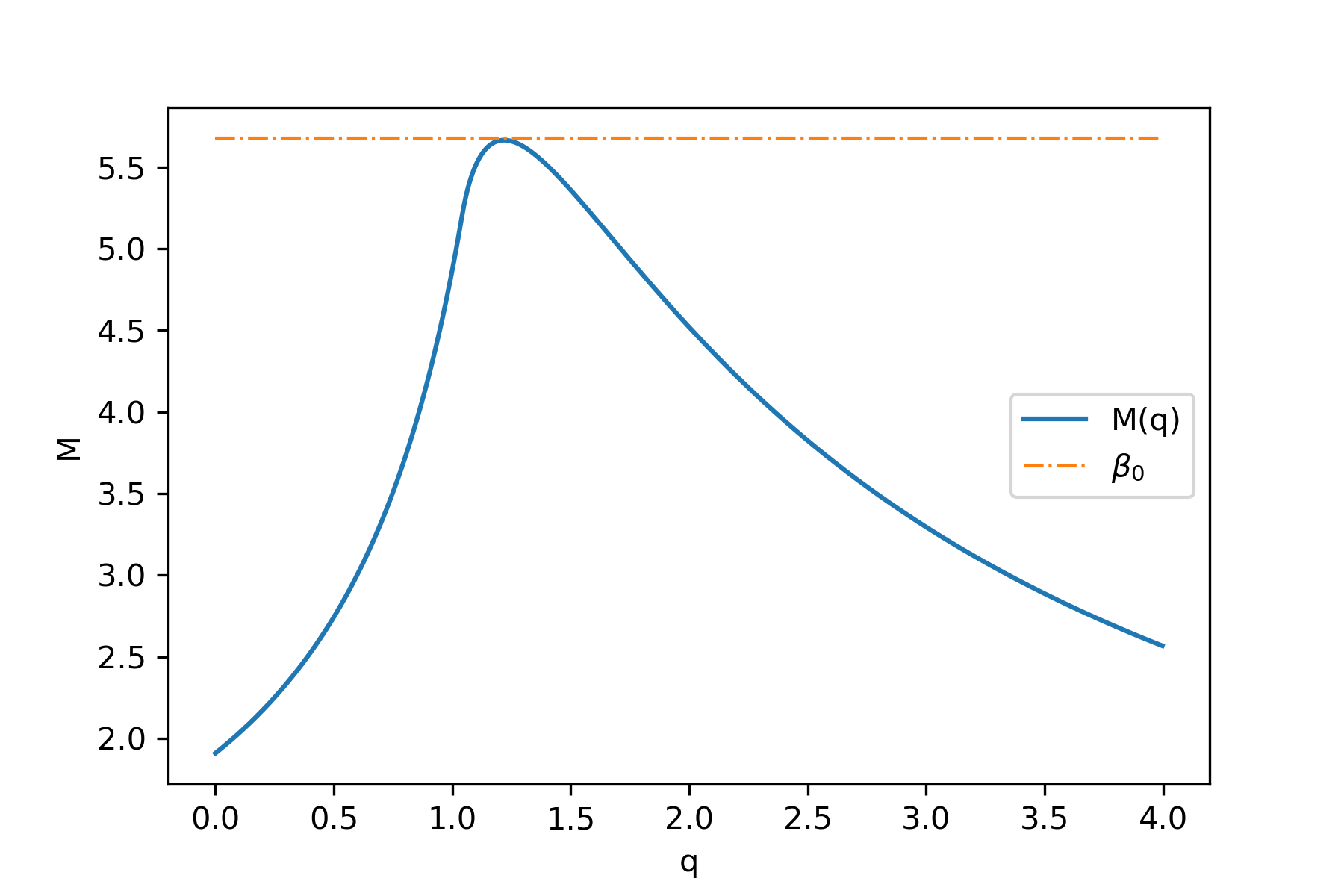}
\caption{Graph of M(q) where its maximum value $\beta_0$ is determined numerically. }
    \label{figmaxq}
\end{center}
\end{figure}
This result obviously only shows that we find a value numerically that 
coincides with $\beta_0$ in \cite{GM97} up to the 
finite precision of the numerical method used for its calculation. 
To prove that the critical value of $\beta$ indeed equals the value determined in \cite{GM97}, 
we need additional considerations. First we describe a somewhat hidden but elementary  
symmetry between $\sin(z)$ and $\sinh(z)$ on the diagonal in the complex plane. 
This may be well-known in other contexts, but since we were not able to 
find an explicit reference we include it here for clarity.   
\begin{lem}\label{symmetrydiag}
Set 
$$
D:=\{z\in\mathbb{C} | \operatorname{Re}(z)=\operatorname{Im}(z)\}\:.
$$
Then for $d\in D$ and $r\in\mathbb{R}$
$$
d\,\sin(rd)\in \mathbb{R} \Longleftrightarrow d\,\sinh(rd)\in \mathbb{R}  
$$
and if the values are real, then  
$$
d\sin(rd) = - d\sinh(rd)\:.
$$
\end{lem}
\begin{proof}
Note that for $d\in D$
$$
\frac{d}{i}={\bar d} \:\text{ and }\:  i\,d=-{\bar d}\:.
$$
We will also use that  sinh is an odd function and that for $z\in\mathbb{C}$ 
$$
\sinh{(\bar z)}=\overline{\sinh(z)} \:\text{ and }\: \sin(z)=\frac{\sinh(i\,z)}{i}\,.
$$
Now assume $d\sin(rd)\in \mathbb{R}$ for $d\in D\,$ and $r\in\mathbb{R}\,$. Then using the above identities
$$
d\sin(rd)=d\,\frac{\sinh(i\,r\,d)}{i}={\bar d}\,\sinh(-\,r\,{\bar d})=-{\bar d}\,\sinh(\,r\,{\bar d})=-\,\overline{d\,\sinh(\,r\,d)}\,
$$
which is the complex conjugate of $-d\,\sinh(rd)\,$ and thus proves the equivalence statement. 
Clearly, if the values are real then the calculation also shows that the claimed sign relationship must hold. 
\end{proof}
Now we can show that the critical $\beta$ determined from the Popov criterion 
coincides with $\beta_0$ obtained from the spectral analysis in \cite{GM97}. 
\begin{prop}\label{equalbeta}
Let 
$$
\Omega_{\text{Nyq}}:=\{\omega>0 \:|\: \operatorname{Im}\bigl[ G_{+}(i\omega)\bigr]=0\}
$$
and 
$$
\Omega_{\text{Pop}}:=
\Bigl\{\omega>0 \:|\: \operatorname{Im}\Bigl[\:\:\operatorname{Re}[ G_{+}(i\omega)] +i\,\omega\, \operatorname{Im}[ G_{+}(i\omega)]\:\:\Bigr]=0\:\Bigr\}\:.
$$
Then
$$
\Omega_{\text{Pop}}=\Omega_{\text{Nyq}}
$$
and 
$$
\Omega_{\text{Pop}}=\{\omega>0 \:|\:  \sqrt{i\omega}\,\sinh(\sqrt{i\omega}\;\pi) \in \mathbb{R}\}=
\{\omega>0 \:|\:  \sqrt{i\omega}\,\sin (\sqrt{i\omega}\;\pi) \in \mathbb{R}\}\,.
$$
Let
$$
\omega_0:= \min\big\{\omega>0 \,\big |\, \sqrt{i\omega}\,\sin (\sqrt{i\omega}\;\pi)\in \mathbb{R}^{+} \big\}.
$$
Then for the constant $\beta_0$ obtained in \cite{GM97} determined by 
$$
\beta_0=\sqrt{i\omega_0}\,sin(\sqrt{i\omega_0}\,\pi) 
$$ 
it holds that 
$$
\beta_0=\sup_{q>0}M(q)\,.
$$
\end{prop}
\begin{proof}
The identity $\Omega_{\text{Pop}}=\Omega_{\text{Nyq}}$ follows from the definition since $\omega=0$ is excluded from both sets. Hence both sets collect the positive values of $\omega$ at which $G_{+}(i\omega)\in\mathbb{R}\,$. 
By the explicit representation of $G_{+}$ given in Remarks \ref{rems} and since for $z\neq0\,$, clearly $z\in\mathbb{R}\Longleftrightarrow \frac{1}{z}\in \mathbb{R}$ we obtain that 
$$
\Omega_{\text{Pop}}=\big\{\omega>0 \,\big |\,  \sqrt{i\omega}\,\sinh(\sqrt{i\omega}\;\pi) \in \mathbb{R}\big\}.
$$
The second characterization of $\Omega_{\text{Pop}}$ that involves $\sin$ instead of $\sinh$ follows 
from Lemma \ref{symmetrydiag} by setting $d=\sqrt{i\omega}$ and $r=\pi\,$.   
The reason why the condition
\begin{equation}\label{charbetaGM97}
\beta_0=\sqrt{i\omega_0}\,\sin(\sqrt{i\omega_0})
\end{equation}
characterizes the constant $\beta_0$ is discussed in \cite{GM97}. 
In fact, for $\beta=\beta_{0}$ the generator of the linear semigroup possesses a 
complex conjugate pair of imaginary eigenvalues $\pm i\,\omega_0$ that gives rise to a 
Hopf bifurcation as the parameter crosses the value $\beta_{0}\,$. 
Finally, that $\beta_{0}\,=\,\sup_{q>0}M(q)$ follows from the fact the Popov curve intersects the real 
axis for $\omega=\omega_0$ and that this intersection is the most negative intersection 
point of the Popov curve with the real axis. 
This point is given by    
$$
G_{+}(i\omega_0)=\frac{1}{\sqrt{i\omega_0}\,\sinh(\sqrt{i\omega_0}\;\pi)}.
$$
Since by \eqref{charbetaGM97} and Lemma \ref{symmetrydiag}
$$
\beta_0=\sqrt{i\omega_0}\,\sin(\sqrt{i\omega_0}\;\pi)=-\sqrt{i\omega_0}\,\sinh(\sqrt{i\omega_0}\;\pi)
$$
we obtain
$$
G_{+}(i\omega_0)= - \frac{1}{\beta_{0}}\,.
$$
Therefore for any $\beta > \beta_{0}$ no half-plane in $\mathbb{C}$ exists that 
contains the real point $-\frac{1}{\beta}$ without intersecting the Popov curve. 
This geometric criterion is implicit in the Popov stability criterion, 
where the parameter $q>0$ varies the slope of the half-plane's boundary. 
Clearly, to be fully rigorous in this argument we would need to discuss the path and the 
asymptotics of the Popov curve in the complex plane in more detail. 
Here we are thus relying to some extent on the qualitative features of the Popov curve that was obtained numerically.   
\end{proof}
The above result establishes a relationship between those values of $\beta\in \mathbb{R}\backslash \{0\}$ that produce a conjugate complex pair of purely imaginary eigenvalues of the operator 
$$
A(\beta):=-\triangle+\beta\gamma_0'\gamma_\pi \:,
$$
studied in \cite{GM99}, and the real values of the Nyquist curve. In fact, the above result also implies that 
$$
\{\beta\in\mathbb{R}\backslash \{0\}  \:|\: \sigma(A(\beta))\cap i\mathbb{R} \neq \emptyset \} 
= \{ -\frac{1}{G_{+}(i\omega)} \:|\: \omega\in \mathbb{R}\backslash \{0\} \text{ and } |\omega|\in \Omega_{\text{Nyq}} \}\;.
$$
After these clarifications we can now formulate a lemma that exploits the possibility to 
control the sign of $R_{\beta,q(\beta)}$ for $\beta\in(0,\beta_0)\,$ to obtain a bound on $V_{\beta,q(\beta)}\,$.
\begin{lem}\label{Elem}
Let $y\in\operatorname{BC}\bigl((0,\infty),\mathbb{R}\bigr)$ be a
solution of the integral equation \eqref{vie} with
$u_0\in\operatorname{H}^1\,$. Let $\beta\in(0,\beta_0)$ and let
$q=q(\beta)>0$ be such that $R_{\beta,q(\beta)}(t)\leq 0$ for $t\geq
0$. Then the function $V=V_{\beta,q(\beta)}$ defined in Lemma
\ref{VandR} satisfies
$$
 V(t)\leq c<\infty\text{ for }t\geq 0,
$$
with a constant $c$ independent of $t\geq 0$.
\end{lem}
\begin{proof}
Since
$$
 V(t)=\int _0^t g_\beta \bigl( y(\tau)\bigr)\bigl[ f(\tau)+q\,
 f'(\tau)\bigr]\, d\tau+W_2(0), 
$$
integration by parts yields
\begin{multline}\label{VRep}
 V(t)=\Bigl( \int_0^t g_\beta \bigl( y(\tau)\bigr)\, d\tau\Bigr) 
 \bigl[ f(t)+q\, f'(t)\bigr]+\\-\int _0^t\bigl[ \int_0^\tau g_\beta
 \bigl( y(\sigma)\bigr)\, d\sigma\bigr] \bigl[ f'(\tau)+q\,
 f''(\tau)\bigr]\, d\tau+W_2(0).\phantom{\hspace{1cm}}
\end{multline}
Since $R_{\beta,q(\beta}(t)\leq 0$ and $W_{\beta,q(\beta)}(t)\geq 0$
for $t\geq 0$, then using \eqref{Lyap} we conclude that, for $t\geq 0$,
$$
 V(t)\geq V(t)+R_{\beta,q(\beta)}(t)=W_{\beta,q(\beta)}\geq
 W_3(t)=\frac{1}{2\pi} \Bigl( \int_0^t g_\beta \bigl( y(\tau)
 \bigr)\, d\tau\Bigr) ^2\geq 0.
$$
It follows that
\begin{equation}\label{VEst2}
 \frac{1}{2\pi}\Bigl(\int_0^t g_\beta \bigl( y(\tau)\bigr)
 \, d\tau\Bigr)^2-V(t)\leq 0,\: t\geq 0\,.
\end{equation}
With
$$
 K(t):=\sup_{\tau\in[0,t]}\big | \int_0^\tau g_\beta \bigl(
 y(\sigma)\bigr)\, d\sigma\big |,
$$
a constant $c>0$ can be found which is independent of $t\geq 0$ and
such that
$$
 \Big | \bigl[ f(t)+q\, f'(t)\bigr] \int_0^t g_\beta \bigl(
 y(\tau)\bigr)\, d\tau \Big |\leq c\, K(t)
$$
and
$$
 \Big | \int _0^t\bigl[ \int_0^\tau g_\beta
 \bigl( y(\sigma)\bigr)\, d\sigma\bigr] \bigl[ f'(\tau)+q\,
 f''(\tau)\bigr]\, d\tau\Big |\leq c\, K(t),
$$
for $t\geq 0$. This is a consequence of the boundedness of $y$, $f$,
$f'$, and $f''$ discussed in Remarks \ref{rems}. Therefore
we infer from \eqref{VRep} that there is a $c>0$ such that
\begin{equation}\label{VEst}
  0\leq V(t)\leq c\, K(t)+W_2(0)\text{ for }t\geq 0.
\end{equation}
Defining
$$
 E(t):=\int_0^t g_\beta \bigl( y(\tau)\bigr)\, d\tau,\: t\geq 0,
$$
we infer from \eqref{VEst2} that
$$
 \frac{1}{2\pi}E^2(t)-c\, \sup_{\tau\in[0,t]}\big | E(\tau)\big
 |-W_2(0)\leq 0,\: t\geq 0.
$$
This finally implies that $E$ is bounded and therefore that $K$ is
bounded, which concludes the proof by \eqref{VEst}.
\end{proof}
\begin{prop}
Let $\beta\in(0,\beta_0)$ and $y\in\operatorname{BC}
\bigl(0,\infty),\mathbb{R}\bigr)$ be a solution of \eqref{vie} with 
$u_0\in \operatorname{H}^1$. Then
$y\in\operatorname{BUC}\bigl([0,\infty),\mathbb{R})\bigr)$.
\end{prop}

\begin{proof}
Since $y$ solves \eqref{vie}, we have that
$$
 y(t)=f(t)-\frac{1}{\pi}\int _0^t g_\beta \bigl( y(\tau)\bigr)\,
 d\tau+\int _0^t a_s(t-\tau) g_\beta \bigl( y(\tau)\bigr)\, d\tau,
$$
for $t\geq 0$ and where $f(t)=\bigl( e^{-tA_N}u_0 \bigr)(\pi)$ is the
forcing function induced by $u_0\,$.
We proceed by verifying that each term is uniformly
continuous. The first term is uniformly continuous by Remark \ref{rems}(b). For 
the second term note that in the proof of Lemma \ref{Elem} it was shown 
that $E$ defined by $E(t)=\int _0^t g_\beta \bigl( y(\tau)\bigr)\, d\tau$ for
$t\geq 0$ is bounded. It is clearly continuous  by assumption (on
$y$). Its derivative is given by $g_\beta\circ y$ which is a bounded
function and hence implies that $E$ is uniformly Lipschitz. 
Finally, the last term can be written as a convolution
$$
 \int _0^t a_s(t-\tau) g_\beta \bigl( y(\tau)\bigr)\, d\tau=\bigl(
 a_s*(g_\beta\circ y)\bigr)(t),\: t\geq 0. 
$$
Since $a_s\in \operatorname{L}^1 \bigl( [0,\infty)\bigr)$ and
$g_\beta\circ y\in \operatorname{L}^\infty \bigl( [0,\infty)\bigr)$,
well-known results about the regularity of convolutions imply the
stated uniform continuity (see e.g. \cite{Ama95} or \cite{Fo84}). 
\end{proof}
Before we can proceed with the proof of Proposition \ref{vieSol20} we
need the following elementary but important result.
\begin{lem}\label{lastLem}
Let $y\in \operatorname{BUC} \bigl( [0,\infty)\bigr)$ and 
$\beta\in (0,\infty)\,$. Then the function
$$
 g_\beta(y)\bigl[ y-\frac{g_\beta(y)}{\beta}\bigr]\in
 \operatorname{BUC}\bigl( [0,\infty)\bigr),
$$
and is positive unless $y=0\,$.
\end{lem}
\begin{proof}
First observe that for $y\neq0$
$$
 g_\beta(y)\bigl[ y-\frac{g_\beta(y)}{\beta}\bigr]=
 \frac{1}{\beta}\tanh(\beta y)\beta y\bigl[ 1-
 \frac{\tanh(\beta y)}{\beta y}\bigr],
$$
which shows that the function is positive unless $y=0$ since
the functions $[z\to z\tanh(z)]$ and $[z\to 1-\frac{\tanh(z)}{z}]$
, for $z=\beta y$, enjoy the same property and $\beta>0$. Next
$$
 g_\beta(y)\bigl[ y-\frac{g_\beta(y)}{\beta}\bigr]=
 \frac{1}{\beta}\tanh(z)\bigl[ z-\tanh(z)\bigr]
$$
shows that the function is the product of bounded and unifornly
continuous functions since $z$ is such and since $\tanh$ has a bounded
derivative and is therefore globally Lipschitz continuous. Indeed this
yields
$$
 \big |\tanh \bigl( z(t)\bigr)- \tanh \bigl( z(s)\bigr)\big |\leq L
 \big | z(t)-z(s)\big |,\: t,s\in [0,\infty),
$$
so that the boundedness and uniform continuity of $z$ is inherited by
$\tanh(z)$. 
\end{proof}
Having now collected all the required technical results in the previous lemmas we can finally 
give the proof of our main proposition on the integral equation \eqref{vie}.

\begin{proof}[{\bf Proof of Proposition \ref{vieSol20}}] $ $\\
Since $\beta\in(0,\beta_0)$ and $y$ solves \eqref{vie}, we can apply
Lemmas \ref{negLem} and \ref{Elem} to find $q(\beta)>0$ such that
$$
 R_{\beta,q(\beta)}(t)\leq 0\text{ for }t\geq 0,
$$
and $c>0$ such that
$$
 c\geq V_{\beta,q(\beta)}(t)\geq W_{\beta,q(\beta)}(t)\geq
 W_{1,\beta}(t)\geq 0,
$$
for $t\geq 0$, i.e. such that
$$
 W_{1,\beta}(t)=\int _0^t g_\beta \bigl( y(\tau)\bigr)
 \bigl[ y(\tau)-\frac{g_\beta \bigl( y(\tau)\bigr)}{\beta}\bigr]
 \, d\tau=:\int _0^t H\bigl( y(\tau)\bigr)\, d\tau\leq c<\infty,
$$
for $t\geq 0$. By Lemma \ref{lastLem} the function $H(y)$ is
non-negative, only vanishes if $y=0$, and is uniformly
continuous. Assume by contradiction that $y(t)\not\to 0$ 
as $t\to\infty$. Then, since $H(\xi)>0$ for $0\neq \xi\in \mathbb{R}$,
there is a sequence $(t_m)_{m\in \mathbb{N}}$ in $\mathbb{R}$ with
$t_m\to\infty$ and a constant $c>0$ such that
$$
 H \bigl( y(t_m)\bigr)\geq 2c\text{ for all }m\in \mathbb{N}. 
$$
Since $H\circ y\in \operatorname{BUC}\bigl([0,\infty)\bigr)$, a
$\delta>0$ can be found such that
$$
 H \bigl( y(t)\bigr)\geq c\text{ for }t\in [t_m-\delta,t_m+\delta ],
$$
and all $m\in \mathbb{N}$. It follows that
$$
 W_1(t_m)=\int _0^{t_m}H \bigl( y(\tau)\bigr)\, d\tau\geq
 \sum_{k=1}^{m-1}\int _{t_k-\delta}^{t_k+\delta}H \bigl(
 y(\tau)\bigr)\, d\tau\geq (m-1)2 \delta c \,, 
$$
which contradicts the boundedness of $W_1$ on $[0,\infty)$ since $m$ can be chosen
arbitrarily large. 
\end{proof}
\section{Global Stability and Convergence Results}
We finally state our main stability and convergence result by explicitly 
determining the global attractor of \eqref{tsEq} for $\beta\in (0,\beta_0)\,$. 
Adopting the terminology in \cite{La91}, we need to make the important 
distinction between the concept of the ``minimal closed global B-attractor" ${\mathcal A}_{\beta}$ 
and the ``minimal closed global attractor" ${\hat{{\mathcal A}}}_\beta$ of our semiflow 
$(\Phi_\beta,\operatorname{H}^1)$ generated by the problem \eqref{tsEq} for each $\beta\in(0,\infty)\,$. 
In general, it holds that the attractor $\hat{{\mathcal A}}_\beta$ defined by 
$$
\hat{{\mathcal A}}_\beta:=\overline{\bigcup\limits_{u \in {H^1}} \omega(u)} 
$$
is a proper subset of the B-attractor ${\mathcal A}_{\beta}$ that is 
defined, in a more implicit way, by 
$$
{\mathcal A}_\beta:=\bigcup\limits_{B \in {\mathcal B}} \omega(B)\:, 
$$
where $\mathcal B$ is the set of all bounded subsets of $H^1\,$.\\
By Theorem 2.3. in \cite{La91} the attractor $\hat{{\mathcal A}}_\beta$ 
coincides withe the set of all stationary points of the 
semiflow if a suitable Ljapunov function exists. 
Using this distinction between different attractor concepts we can now formulate our main result.   
\begin{thm}\label{mainresult}
For $\beta\in(0,\beta_0)$ the minimal closed global attractor $\hat{{\mathcal A}}_\beta\subset H^1$ of the continous semiflow $(\Phi_\beta,\operatorname{H}^1)$ 
generated by the nonlinear and nonlocal problem \eqref{tsEq} is given by 
$$
\hat{{\mathcal A}}_\beta=\{ 0 \}\,,
$$
where $\beta_0\approx 5.6655$ is the constant determined in \cite{GM97}.\\
For $\beta \in (0,\frac{4}{\pi})$ the attractor $\hat{{\mathcal A}}_\beta$ and the B-attractor ${\mathcal A}_\beta$ coincide, i.e. 
$$
\mathcal{A}_\beta=\hat{{\mathcal A}}_\beta=\{ 0 \}\,.
$$
\end{thm}
\begin{proof}
Let $u\in H^1\,$ arbitrary. Then by Lemma \ref{VoltInt1} the trace of the orbit 
$\gamma_\pi(\Phi_\beta(t,u))$ solves the Volterra integral equation \eqref{vie}. 
By Proposition \ref{vieSol20} the assumption $\beta\in(0,\beta_0)$ implies that as $t \to\infty$
$$
\gamma_\pi(\Phi_\beta(t,u))\to 0\,\text{ in }\mathbb{R}\:. 
$$
Finally, by Proposition \ref{ConvEquiv} we conclude that 
$$
\Phi_\beta(t,u)\to 0\,\text{ in } H^1\: \text{ as } t\to\infty \:,
$$
which implies that the omega-limit set $\omega(u)=\{0\}\,$. Therefore $\hat{{\mathcal A}}_\beta=\{ 0 \}\,$. \\ 
For $\beta\in(0,\frac{4}{\pi})$ we have shown in Remark \ref{Ljap2} that a Ljapunov function for the semiflow exists. 
By Theorem 2.3. in \cite{La91} the B-attractor $\mathcal{A}_\beta$ 
consists of complete trajectories connecting stationary states of the semiflow $(\Phi_\beta,\operatorname{H}^1)\,$. 
Since we know that $\hat{{\mathcal A}}_\beta=\{ 0 \}$ by the proof of the first statement and since non-trivial complete orbits connecting $\{0\}$ with 
itself are excluded by the existence of the Ljapunov function, we conclude that $\mathcal{A}_\beta=\hat{{\mathcal A}}_\beta=\{ 0 \}\,$.  
\end{proof}
We conclude with some remarks on open problems and the extension of the results to a more general 
class of nonlinearities  $g_\beta\,$, which so far we have assumed to be equal to $\tanh(\beta\cdot)$ for simplicity.
\begin{rems}
{\bf (a)}
Assume that for each $\beta\in(0,\infty)$ the function $g_\beta: \mathbb{R} \to \mathbb{R}$ has the following properties:
\begin{itemize}
    \item[(i)] $g_\beta$ is bounded and globally Lipschitz continous.
    \item[(ii)] $g_\beta(0)=0$ and $g'_\beta(0)=\beta\,$.
    \item[(iii)] $g_\beta(w) (w - \frac{g_\beta(w)}{\beta})>0$ for $w\neq 0\,$.
\end{itemize}
Then Theorem \ref{mainresult} remains true if $\tanh(\beta\cdot)$ 
in \eqref{tsEq} is replaced by a different nonlinear function 
$g_\beta(\cdot)$ with the above properties. 
This follows from an inspection of the assumptions made on $g_\beta\,$. 
In particular, in the proof of Proposition \ref{vieSol20} and in Remark \ref{Ljap2}.\\[0.1cm]
{\bf (b)}     
We conjecture that $\mathcal{A}_\beta=\hat{{\mathcal A}}_\beta=\{ 0 \}$ for $\beta\in(0,\beta_0)\,$. 
Since we were not able to find a Ljapunov function for $\beta\geq\frac{4}{\pi}$, 
this natural conjecture may require a different proof that we don't have at this time.\\[0.1cm] 
{\bf (c)} By the existence of a non-trivial periodic solution for $\beta\in(\beta_0,\beta_0+\varepsilon)\,$ for 
some $\varepsilon > 0$ we know that $\{0\}$ is a proper subset of $\mathcal{A}_\beta$ for $\beta\in(\beta_0,\beta_0+\varepsilon)\,$. 
We conjecture that this is true for $\beta\in (\beta_0,\infty)\,$.\\[0.1cm] 
{\bf (d)}
A proof of the two conjectures above would imply that $(0,\beta_0)$ is the maximal open set in $[0,\infty)$ 
where $\mathcal{A}_\beta=\hat{{\mathcal A}}_\beta=\{ 0 \}\,$. 
\end{rems}

\bibliography{lite.bib}
\end{document}